\documentclass[11pt]{amsart}
\usepackage{amsmath,amsfonts,amssymb,amsthm,epsfig,color}
\usepackage{hyperref} 

\newtheorem{theorem}{Theorem}[section]
\newtheorem{corollary}[theorem]{Corollary}
\newtheorem{lemma}[theorem]{Lemma}

\newtheorem{remark}[theorem]{Remark}

\newtheorem{proposition}[theorem]{Proposition}

\newcommand{\R}{\mathbb{R}}
\newcommand{\N}{\mathbb{N}}

\numberwithin{equation}{section}

\renewcommand{\le}{\leqslant}
\renewcommand{\leq}{\leqslant}
\renewcommand{\ge}{\geqslant}
\renewcommand{\geq}{\geqslant}

\def\H{\mathcal{H}}
\def\N{\mathbb N}
\def\R{\mathbb R}
\def\S{\mathcal S}

\def\I{\mathcal J}

\def\eps{\varepsilon}

\def\a{\alpha}

\def\ov{\overline}

\def\bal{\begin{aligned}}
\def\eal{\end{aligned}}

\def\F{{\mathcal{F}_{s,t}}}
\def\Feps{{\mathcal{F}_{s,t}^{\,\varepsilon}}}
\def\G{\mathcal{G}}
\def\l{\lambda}

\title[Nonlocal isoperimetric problems]{Nonlocal isoperimetric problems}

\author[A. Di Castro]{Agnese Di Castro}
\email[Agnese Di Castro]{\href{mailto:dicastro@mail.dm.unipi.it}{dicastro@mail.dm.unipi.it}}
\author[M. Novaga]{Matteo Novaga}
\email[Matteo Novaga]{\href{mailto:novaga@dm.unipi.it}{novaga@dm.unipi.it}}
\author[B. Ruffini]{Berardo Ruffini}
\email[Berardo Ruffin]{\href{mailto:berardo.ruffini@ujf-grenoble.fr}{berardo.ruffini@ujf-grenoble.fr}}
\author[E. Valdinoci]{Enrico Valdinoci}
\email[Enrico Valdinoci]{\href{mailto:enrico.valdinoci@wias-berlin.de}{enrico.valdinoci@wias-berlin.de}}

\address[A. Di Castro, M. Novaga]
{Dipartimento di Matematica, Universit\`a di Pisa
\newline  Largo Bruno Pontecorvo,~5
\\ 56127 Pisa, Italy}
\address[B. Ruffini]
{Universit\'e Joseph Fourier 
\\ BP 53
\\ 38041, Grenoble cedex 39, France}
\address[E. Valdinoci]
{Weierstrass Institute
\\ Mohrenstrasse~39
\\ 10117 Berlin, Germany}

\begin{document}

\begin{abstract}
We characterize the volume-constrained minimizers of a nonlocal free energy 
given by the difference of the $t$-perimeter and 
the $s$-perimeter, with $s$ smaller than $t$. 
Exploiting the quantitative fractional isoperimetric inequality,
we show that balls are the unique minimizers if the volume is sufficiently small,
depending on $t-s$, while the existence vs. nonexistence of minimizers 
for large volumes remains open.
We also consider the corresponding isoperimetric problem 
and prove existence and regularity of minimizers for all $s,\,t$.
When $s=0$ this problem reduces to the fractional isoperimetric problem,
for which it is well known that balls are the only minimizers.
\end{abstract}

\maketitle
\tableofcontents

\section{Introduction}\label{section0}
In this paper we deal with two nonlocal isoperimetric problems, which are closely related one with the other. 
To introduce them, we recall the definition and
some properties of the fractional perimeter. Given a number $\alpha\in(0,1)$, for a measurable set $E\subset\R^N$,
the fractional
perimeter $P_\alpha(E)$ is defined as the (squared) $H^{\alpha/2}$-seminorm of the
characteristic function of $E$, that is,
\[
P_\alpha(E):=\left[
\chi_E\right]^2_{H^{\alpha/2}}=\int_{\R^N}\int_{\R^N}\frac{|\chi_E(x)-\chi_E(y)|^2}{|x-y|^{N+\alpha}}dx\,
dy=\int_E\int_{E^c}\frac{
dx\,dy}{|x-y|^{N+\alpha}}.
\]
The notion of fractional perimeter has been introduced in~\cite{Vis, CRS10}
and it has been extensively
studied in several recent papers 
(see for instance~\cite{FMM11, SV12, SV13, CV13, labandadei5, delpino}
and references therein). In particular, according \cite[Theorem 1]{CV11}
(see also ~\cite{Bo, davila, ADM11}),
we have that the fractional perimeter~$P_\alpha$, if suitably renormalized, approaches
the classical perimeter~$P$ as~$\alpha\nearrow1$. More precisely, if
$\partial E$ is of class $C^{1,\gamma}$ for some $\gamma>0$, we have 
\begin{equation}\label{s1}
\lim_{\alpha \to 1^-} (1-\alpha) P_\alpha (E)=N\omega_{N} P(E),
\end{equation}
where $\omega_N$ denotes the volume of the $N$-dimensional ball of radius $1$.
On the other hand, 
the fractional perimeter~$P_\alpha$ approaches the Lebesgue measure~$|\cdot|$
as~$\alpha\searrow0$, that is, 
\begin{equation}\label{s0}
\lim_{\alpha \to 0^+} \alpha P_\alpha (E)=N \omega_N |E|,
\end{equation}
if $P_{\bar{\alpha}}(E)<+\infty$ for some $\bar{\alpha}>0$ (see \cite{Maz} and~\cite[Corollary 2.6]{DFPV}).

\smallskip

In the first part of the paper
we investigate the minimum problem:
\begin{equation}\label{problem}
\min_{|E|=m}\F (E)\qquad m\in (0,+\infty)\,,
\end{equation}
where 
\begin{equation}\label{Fst}
\F(E):=
\begin{cases}
(1-t)P_t(E)-sP_s(E) & \text{if $0<s<t<1$}\\
 \quad& \quad\\
N\omega_{N}P(E)-sP_s(E) & \text{if $0<s<t=1$}\\
\quad& \quad\\
(1-t)P_t(E)-N\omega_{N}|E| & \text{if $0=s<t<1$}\\
\quad &\quad\\
N\omega_{N}P(E)-N\omega_{N}|E|&\text{if $s=0$ and $t=1$}.
\end{cases}
\end{equation}

\noindent
Notice that thanks to \eqref{s0} and \eqref{s1}, for all $s,t\in(0,1)$ 
we have 
\begin{equation}\label{F.F}
\F (E) \underset{t\to1}{\to}
\mathcal F_{s,1}(E)\underset{s\to0}{\to} \mathcal F_{0,1}(E)
\quad \text{ and }\quad
\F(E) \underset{s\to0}{\to} \mathcal F_{0,t}(E)
\underset{t\to1}{\to}\mathcal F_{0,1}(E)\,,
\end{equation}
that is, $\F$ depends continuously on $s,t\in [0,1]$, with $s<t$.

Problem \eqref{problem} is reminiscent of recent results about isoperimetric problems with nonlocal
competing term arising in mathematical physics, where the  functionals take the form
\[
 \mathcal F=P+\mathcal{NL}
\]
being $P$ the perimeter and $\mathcal{NL}$ the nonlocal term,
see for instance \cite{KMP13_I,KMP13_II,CS13,GNR13,FJ,labandadei5,BC14,J}.
We mention in particular the works by Kn\"upfer and Muratov
\cite{KMP13_I,KMP13_II} where the authors consider the case
where the nonlocal term is given by a Coulombic potential. 

In our framework, the energy in~\eqref{F.F}
presents a competing effect between the term~$P_t$
which has the tendency to ``aggregate'' the sets into balls,
and~$P_s$, which acts in the opposite way. We will see that, at small scales,
the aggregating effect is predominant, but this does
not occur at large scales.

More precisely, as a first result we show that minimizers exist 
and are regular at least for small volumes. 

\begin{theorem}\label{thexist}
For any $0\le s<t\le 1$, there exists 
$\bar m_0=\bar m_0(N, t-s)>0$ such that for all $m\in(0,\bar m_0)$, problem 
\eqref{problem} has a minimizer $F\subset \mathbb{R}^N$.
Moreover $F$ is bounded with boundary of class $C^{1,\beta}$, for some $\beta=\beta(N,t-s)\in (0,1)$,
outside a closed singular set of Hausdorff dimension 
at most $N-2$ (respectively $N-8$ if $t=1$).
\end{theorem}

Thanks to the fractional isoperimetric inequality in a quantitative form,
we also show that the 
the minimizer found in Theorem~\ref{thexist}
is necessarily a ball, if the volume $m$ is sufficiently small.

\begin{theorem}\label{thmain}
For any $0\leq s<t\leq 1$ and $\bar m_0$ as in Theorem~\ref{thexist}, there exists $\bar m_1=\bar m_1(N, t-s)\in (0,\bar m_0]$ 
such that for all $m\in (0,\bar m_1)$, the only minimizer 
of problem \eqref{problem} is given by the ball of measure $m$.
\end{theorem}

We stress that our estimates, similarly to those in \cite{labandadei5}, 
depend only on a lower bound on the difference $t-s$,
and pass to the limit as $s\to 0$ and $t\to 1$
(as a matter of fact, the normalizing constants
appearing in~\eqref{Fst} has exactly the purpose of
making our estimates stable as $s\to 0$ and $t\to 1$).

Moreover, as far as we know, our results are new even in the case~$t=1$.

We also point out that we do not know if a minimizer exists for any volume $m$. 
However, we show that {\em a minimizer cannot be a ball if $m$ is
large enough} (see Theorem \ref{prounstable}), so the minimization problem
can be in general quite rich.

\smallskip

The second problem we consider is the following 
generalized isoperimetric problem:
\begin{equation}\label{isoperintro}
\min_{E\subset\R^N}\,\widetilde\F(E)
\qquad 0\le s<t\le 1\,,
\end{equation}
where
\[
\widetilde\F(E):=
\begin{cases}
\dfrac{\left((1-t)P_t(E)\right)^{N-s}}{\left(sP_s(E)\right)^{N-t}} & \text{if $0<s<t<1$}\\
 \quad& \quad\\
\dfrac{(N\omega_{N}P(E))^{N-s}}{\left(sP_s(E)\right)^{N-1}} & \text{if $0<s<t=1$}\\
\quad& \quad\\
\dfrac{(1-t)P_t(E)^{N}}{(N\omega_{N}|E|)^{N-t}} & \text{if $0=s<t<1$}\\
\quad &\quad\\
N\omega_{N}\frac{P(E)^{N}}{|E|^{N-1}}&\text{if $s=0$ and $t=1$}.
\end{cases}
\]
Again, thanks to \eqref{s0} and \eqref{s1} we see that
$$
\widetilde{\F} (E) \underset{t\to1}{\to}  
\widetilde{\mathcal F_{s,1}}(E)\underset{s\to0}{\to}  \widetilde{\mathcal F_{0,1}}(E) \quad \text{ and }\quad
\widetilde{\F}(E) \underset{s\to0}{\to}
\widetilde{\mathcal F_{0,t}}(E)\underset{t\to1}{\to}
\widetilde{\mathcal F_{0,1}}(E).
$$
Since, for $s=0$ and $t=1$, problem \eqref{isoperintro} reduces to the classical isoperimetric one, we can
think to it as a generalized isoperimetric problem for fractional perimeters. 

We now state our main result about problem \eqref{isoperintro}.
\begin{theorem}\label{mainisopintro}
For any $0\le s<t\le 1$, there exists a nontrivial minimizer 
$E_{s,t}$ of problem \eqref{isoperintro}. Moreover 
$E_{s,t}$ is bounded and has boundary of class $C^{1,\beta}$,
for some $\beta=\beta(N,t-s)\in (0,1)$, 
outside a closed singular set of Hausdorff dimension 
at most $N-2$ (respectively $N-8$ if $t=1$).
\end{theorem}

We point out that, for $s=0$, the problem reduces to the fractional 
isoperimetric problem, for which it is known that the ball 
is the unique minimizer \cite{FS08} (see also \cite{FMM11} for a quantitative 
version of this result). However, we do not know if the ball is still a
minimizer of problem \eqref{isoperintro} for $s>0$.

\smallskip

The paper is organized as follows:
in Section \ref{section1} we recall some general 
properties of the fractional perimeters and, more generally, 
of the fractional Sobolev seminorms. 
In Sections \ref{section2}--\ref{section5} 
we deal with problem \eqref{problem}. Section
\ref{section2} contains the main tools exploited later to prove 
Theorems \ref{thexist} and \ref{thmain}. The
cornerstone of the section is an optimality criterion (see Proposition
\ref{nonoptimality})
which entails density estimates for
minimizers (see Proposition \ref{densityestimate}) 
and the fact that minimizers must be close to a ball, 
if the volume is small enough (see Lemma \ref{deficit}). An
elementary, but important result is then provided 
by Proposition \ref{necessario}, stating that any minimum must be
necessary bounded and, if $t=1$ (that is, $\mathcal{F}_{s,1}=N\omega_NP-sP_s$), 
also essentially connected.
Section \ref{section3} contains Theorem \ref{PRE-thexist}, 
which solves the existence
part of Theorem \ref{thexist}, while in Section \ref{section4} we prove 
that any minimizer has smooth boundary, out of a closed singular set. 
Then, in Section \ref{section5} we show that, if the volume 
$m$ is below a certain threshold $\bar m_1>0$, 
the ball is the unique minimizer for problem \eqref{problem}.
Eventually, in Section \ref{section6}, 
we deal with problem \eqref{isoperintro}. The main result here is
given by Theorem \ref{mainisopintro}, where we show the existence 
and regularity of minimizers.

\section{General properties of fractional perimeters}\label{section1}
Before starting to prove some properties of fractional perimeters it is convenient to fix some notation which will be used throughout the rest of the paper. Firstly, notice that we will denote by $c_N$ a general positive constant depending only on the dimension $N$ and by $c_0$ a positive constant depending on $N$ and $\delta_0$ a fixed quantity such that $0<\delta_0\leq t-s$, which will not necessarily be the same at different occurrences and which can also change from line to line; special constants will be denoted by $c_1$, $c_2$,....
Relevant dependences on parameters will be emphasized by using parentheses.

As customary, we denote by 
$B(x_0,R):=\{x\in \mathbb{R}^N : |x-x_0|<R\}$
the open ball centered in $x_0 \in \mathbb{R}^N$ with radius $R > 0$. We shall use the shorter notation $B=B(0,1)$, with $|B(0,1)|=\omega_N$. Moreover, when not important and clear from the context, we shall denote by $B_m$ the ball of volume $m$, that is of radius $R=(m/|B(0,1)|)^{1/N}$.

Finally, as usual, given two sets $E$ and $F$ of $\mathbb{R}^N$, we denote the symmetric difference between $E$ and
$F$ as $E \Delta F = (E \setminus F) \cup (F \setminus E )$.

\vspace{5mm}

We begin by a simple result.

\begin{lemma}\label{rmk1}
Let $E=E_1\cup E_2$ a subset of $\R^N$ with $|E_1\cap E_2|=0$. Then
\begin{equation}\label{eqsum}
 P_\a(E)=P_\a(E_1)+P_\a(E_2)-
2\int_{E_1}\int_{E_2}\frac{dx\,dy}{|x-y|^{N+\a}}.
\end{equation}
In particular
\begin{equation}\label{R2}
P_\a(E)\le P_\a(E_1)+P_\a(E_2).
\end{equation} 
\end{lemma}
\begin{proof}
 Let us denote by $\chi_E$ the characteristic function of the set $E$. We have
 \[
  \begin{aligned}
P_\a(E)&= \int_{\R^N}\int_{\R^N}\frac{(\chi_E(x)-\chi_E(y))^2}{|x-y|^{N+\a}}dx\,dy\\
&=\int_{\R^N}\int_{\R^N}\frac{(\chi_{E_1}(x)+\chi_{E_2}(x)-\chi_{E_1}(y)-\chi_{E_2}(y))^2}{|x-y|^{N+\a}}dx\,dy\\ 
&=\int_{\R^N}\int_{\R^N}\frac{(\chi_{E_1}(x)-\chi_{E_1}(y))^2+(\chi_{E_2}(x)-\chi_{E_2}(y))^2}{|x-y
|^{N+\a}} \\
&+\int_{\R^N}\int_{\R^N}\frac{2(\chi_{E_1}(x)-\chi_{E_1}(y))(\chi_{E_2}(x)-\chi_{E_2}(y))} { |x-y|^ { N+\a } }
dx\,dy\\
&=P_\a(E_1)+P_\a(E_2)-2\int_{E_1}\int_{E_2}\frac{dx\,dy}{|x-y|^{N+\a}}.
  \end{aligned}
 \]

\end{proof}

For further use, we also prove the following
interpolation estimate (by reasoning as in~\cite[Proposition 4.2 and Corollary $4.4$]{bralinpar}):

\begin{lemma}\label{int_est}
For any $E\subset\R^N$ and $0< s<t< 1$ there holds
\begin{equation}\label{basicestimates}
P_s(E)\le c_N\,\frac 1s\left( 1-\frac st\right)^{-1}|E|^{1-\frac s t}(1-t)^\frac s tP_t(E)^\frac s t.
\end{equation}
\end{lemma}

\begin{proof}
We reason as in \cite[Prop. 4.2]{bralinpar}.
Letting $u=\chi_E$, we can write
\begin{eqnarray*}
P_s(E)&=&\int_{\R^N}\int_{\R^N}\frac{|u(x+h)-u(x)|}{|h|^{N+s}}\,dxdh
\\
&=& \int_{|h|<1}\int_{\R^N}\frac{|u(x+h)-u(x)|}{|h|^{N+s}}\,dxdh
\\
&&+ \int_{|h|\ge 1}\int_{\R^N}\frac{|u(x+h)-u(x)|}{|h|^{N+s}}\,dxdh
\,=:\, I_1+I_2.
\end{eqnarray*}
We recall that, by \cite[Lemma A.1]{bralinpar} (see also \cite{Maz}),
there exists a constant $c_N$ such that 
\begin{equation}\label{brA1}
\int_{\R^N}\frac{|u(x+h)-u(x)|}{|h|^{t}}\,dx \le
c_N\, (1-t) P_t(E)
\,,
\end{equation}
for all $|h|>0$. We then estimate
\begin{eqnarray}\nonumber
I_1 &=& \int_{|h|<1}\int_{\R^N}\frac{|u(x+h)-u(x)|}{|h|^{N+s}}\,dxdh
\\ \label{stimaone}
&\le& c_N (1-t) P_t(E) \int_{|h|<1} \frac{1}{|h|^{N-(t-s)}}\,dh
\\\nonumber
&=& c_N\frac{1-t}{t-s} P_t(E),
\end{eqnarray}
and
\begin{eqnarray}\nonumber
I_2 &=& \int_{|h|\ge 1}\int_{\R^N}\frac{|u(x+h)-u(x)|}{|h|^{N+s}}\,dxdh
\\ \label{stimatwo}
&\le& 2 |E| \int_{|h|\ge 1} \frac{1}{|h|^{N+s}}\,dh
\\\nonumber
&=& \frac{2N\omega_N}{s}|E|.
\end{eqnarray}
Putting together \eqref{stimaone} and \eqref{stimatwo}
we then get
\begin{equation}\label{stimathree}
P_s(E)\le c_N\frac{1-t}{t-s} P_t(E) + \frac{2N\omega_N}{s}|E|.
\end{equation}
If we evaluate \eqref{stimathree} on the set $\lambda E$, 
with $\lambda>0$, we obtain
\[
\lambda^{N-s}P_s(E)\le 
c_N\frac{1-t}{t-s} \lambda^{N-t}P_t(E) 
+ \lambda^N \frac{2N\omega_N}{s}|E|,
\]
that is,
\begin{equation}\label{stimafour}
\lambda^{t-s}P_s(E)-\lambda^t\frac{2N\omega_N|E|}{s}
\le \frac{c_N(1-t)}{t-s}P_t(E).
\end{equation}
The expression at the left-hand side of \eqref{stimafour}
reaches its maximum at
\[
\lambda = \left(\frac{s(t-s)P_s(E)}{2N\omega_Nt|E|}\right)^\frac 1 s.
\]
Substituting this value of $\lambda$ into \eqref{stimafour}
we get \eqref{basicestimates}.
\end{proof}

\begin{remark}\rm
If we let $t\to1^-$ in \eqref{basicestimates},
we recover the estimate in \cite[Cor. 4.4]{bralinpar}:
\begin{equation}\label{basicestimate}
P_s(E)\le \frac{c_N}{s(1-s)}|E|^{1-s}P(E)^s.
\end{equation}
Indeed the proof of Lemma~\ref{int_est} extends to the case $t=1$, by substituting $(1-t)P_t(E)$ with $P(E)$ in the
right hand side of \eqref{brA1}.
\end{remark}

\noindent
We show now  a version
of the local fractional isoperimetric inequality.
For this, we recall that the fractional perimeter of a set~$E$ in
a bounded set~$\Omega$ is defined by
\begin{equation}\label{67-0}
P_\a(E,\Omega):= \int_{E\cap\Omega}\int_{\R^N\setminus E}
\frac{dx\,dy}{|x-y|^{N+\alpha}}+
\int_{\Omega\setminus E}\int_{E\setminus \Omega}
\frac{dx\,dy}{|x-y|^{N+\alpha}}
..\end{equation}
With this setting, we have a variant of Lemma~\ref{rmk1} as follows:

\begin{lemma}
Let~$\Omega_1$ and~$\Omega_2$ be disjoint bounded sets. Then
\begin{equation}\label{U35}
P_\a(E,\Omega_1)+P_\a(E,\Omega_2)\le P_\a(E,\Omega_1\cup\Omega_2)
+2\int_{\Omega_1}\int_{\Omega_2}\frac{dx\,dy}{|x-y|^{N+\alpha}}
..\end{equation}
\end{lemma}

\begin{proof} We use~\eqref{67-0} (omitting the integrands for simplicity)
to compute
\begin{eqnarray*}
&& P_\a(E,\Omega_1\cup\Omega_2)
-P_\a(E,\Omega_1)-P_\a(E,\Omega_2)
\\ &=& \int_{E\cap(\Omega_1\cup\Omega_2)}\int_{\R^N\setminus E}
+
\int_{(\Omega_1\cup\Omega_2)\setminus E}\int_{E\setminus (\Omega_1\cup\Omega_2)}
\\ &&-
\int_{E\cap\Omega_1}\int_{\R^N\setminus E}
-
\int_{\Omega_1\setminus E}\int_{E\setminus \Omega_1}
-
\int_{E\cap\Omega_2}\int_{\R^N\setminus E}
-
\int_{\Omega_2\setminus E}\int_{E\setminus \Omega_1}
\\ &=&
\int_{E\cap\Omega_1}\int_{\R^N\setminus E}+\int_{E\cap\Omega_2}\int_{\R^N\setminus E}
+
\int_{\Omega_1\setminus E}\int_{E\setminus (\Omega_1\cup\Omega_2)}
+ \int_{\Omega_2\setminus E}\int_{E\setminus (\Omega_1\cup\Omega_2)}
\\ &&-
\int_{E\cap\Omega_1}\int_{\R^N\setminus E}
-
\int_{\Omega_1\setminus E}\int_{E\setminus \Omega_1}
-
\int_{E\cap\Omega_2}\int_{\R^N\setminus E}
-
\int_{\Omega_2\setminus E}\int_{E\setminus \Omega_1}
\\ &=&
\int_{\Omega_1\setminus E}\int_{E\setminus (\Omega_1\cup\Omega_2)}
+ \int_{\Omega_2\setminus E}\int_{E\setminus (\Omega_1\cup\Omega_2)}
-\int_{\Omega_1\setminus E}\int_{E\setminus \Omega_1}
-
\int_{\Omega_2\setminus E}\int_{E\setminus \Omega_1}\\&=&
-\int_{\Omega_1\setminus E}\int_{(E\setminus \Omega_1)\cap\Omega_2}
-\int_{\Omega_2\setminus E}\int_{(E\setminus \Omega_2)\cap\Omega_1}
.
\end{eqnarray*}
This implies~\eqref{U35}.
\end{proof}

Then, we have the following local fractional isoperimetric inequality:
\begin{lemma}\label{lemfracisoploc}
 Let $\Omega$ be a open bounded set with Lipschitz boundary and let $E\subseteq\R^N$ such that $|E\cap
\Omega|<|\Omega|/2$. Then there exists a constant $C=C(|\Omega|,N,\a)$ such that
\begin{equation}\label{fracisoploc}
 P_\a(E,\Omega)\ge C |E\cap\Omega|^{\frac{N-\a}{N}}.
\end{equation}
\end{lemma}
\begin{proof}
The case $t=1$ is classical. For its proof we refer to \cite[Section $II.1.6$]{M12}.
 We begin by recalling the Poincar\'e-type inequality for fractional Sobolev spaces (see for instance
\cite[Equations $(2)$ and $(3)$]{Bo}: for any $p\ge1$ and $\a\in(0,1)$, given a function $f\in L^p(\Omega)$ we have that  
 \begin{equation}\label{poincare}
\int_\Omega\int_\Omega\frac{|f(x)-f(y)|^p}{|x-y|^{N+\a p}}\ge
C(N,\a,p,|\Omega|)\,\|f-f_\Omega\|_{L^q(\Omega)},
 \end{equation}
where
\[
 f_\Omega=\frac{1}{\Omega}\int_\Omega |f|\,dx
\]
and
\begin{equation}\label{indici}
 \frac1q=\frac1p-\frac \a N.
\end{equation}
By applying \eqref{poincare} with $p=1$, $\a\in(0,1)$ and $f=\chi_E$, and by the very definition of $ P_\a(E)$ we
get
that
\[
 \begin{aligned}
  2P_\a(E,\Omega)&\ge \int_\Omega\int_\Omega \frac{|\chi_E(x)-\chi_E(y)|}{|x-y|^{N+\a}}\\
   &\ge C(N,\a,|\Omega|)\left(\int_\Omega
\left|\chi_E(x)-\frac{|E\cap\Omega|}{|\Omega|}\right|^q\,dx\right)^{1/q}\\
   &=C(N,\a,|\Omega|)\left[|E\cap\Omega|\left(1-\frac{|E\cap\Omega|}{|\Omega|}\right)^q+|\Omega\setminus
E|\left(\frac{|E\cap\Omega|}{|\Omega|}\right)^q \right]^{1/q}\\
   &\ge C(N,\a,|\Omega|)|E\cap\Omega|^{1/q}\left(1-\frac{|E\cap\Omega|}{|\Omega|}\right)\\
  &\ge \frac{C(N,\a,|\Omega|)}{2} |E\cap\Omega|^{1/q}.
 \end{aligned}
\]
Since, by \eqref{indici}, $q=N/(N-\a)$, the proof is concluded.
\end{proof}

\noindent
Beside the local fractional isoperimetric inequality \eqref{fracisoploc}, we recall from \cite{FS} the
standard (fractional) one: if $0<t_0\le \a\le 1$ then it holds (if $|E|<+\infty$)
\begin{equation}\label{fracisop}
(1-\a)P_\a(E)\ \ge\ c(N,t_0) |E|^\frac{N-\a}{N},\qquad \text{ where $c(N,t_0)=\frac{c_N}{t_0}$},
\end{equation}

We now recall some basic facts on hypersingular Riesz operators on the sphere,
following \cite[pp. 4-5]{labandadei5} (see also \cite[pp. 159-160]{Sam}).
We denote by $\S_k$ the space of spherical harmonics of degree $k$, and
by $d(k)$ the dimension of $\S_k$. For $\alpha\in (0,1)$ we also let 
$\I_\alpha$ be the operator defined as 
\[
\I_\alpha u(x) = 2\,{\rm p.v.} \int_{\partial B}	\frac{u(x)-u(y)}{|x-y|^{N+\alpha}}\,d\H^{N-1}(y)\qquad {\rm for }\ u\in
C^2(\partial B),
\]
(with the symbol p.v. we mean that the integral is considered in the Cauchy principal value sense)
and we let $\lambda_k^\alpha$ be the $k^{\rm th}$ eigenvalue of $\I_\alpha$, that is,
\[
\I_\alpha Y = \lambda_k^\alpha Y \qquad {\rm for\ any\ }Y\in \S_k.
\]
We then have $\lambda_k^\alpha\to +\infty$ as $k\to +\infty$, and
\[
\lambda_0^\alpha=0\qquad \lambda_{k+1}^\alpha>\lambda_k^\alpha\qquad \forall k\in\mathbb N\cup\{0\}.
\]
If we let $\{Y^i_k\}_{i=1}^{d(k)}$ be an orthonormal basis of $\S_k$
in $L^2(\partial B)$, and denote by 
\[
a_k^i(u) := \int_{\partial B}u\, Y^i_k\,d\H^{N-1},
\]
the Fourier coefficients of $u\in L^2(\partial B)$ corresponding to $Y^i_k$,
we have
\begin{eqnarray}\label{armonic}
\nonumber[u]^2_{H^\frac{1+\alpha}{2}(\partial B)} &:=& \int_{\partial B}\int_{\partial
B}\frac{|u(x)-u(y)|^2}{|x-y|^{N+\alpha}} \,d\H^{N-1}(x)\,d\H^{N-1}(y)
\\
\nonumber&=& \int_{\partial B}u\, \I_\alpha u\,d\H^{N-1} 
\\
&=& \sum_{k=0}^\infty \sum_{i=0}^{d(k)}\lambda_k^\alpha \,a_k^i(u)^2.
\end{eqnarray}

\begin{proposition}\label{prola}{\rm (\cite[Proposition 2.3]{labandadei5})} We have
\begin{eqnarray*}
\lambda_k^\alpha\geq \lambda_1^\alpha = \alpha(N-\alpha)\frac{P_\alpha(B)}{P(B)} 
\ge \frac{1}{c_N(1-\alpha)}\,.
\end{eqnarray*}
\end{proposition}

\begin{proposition}
Let $u\in H^{\frac{1+t}{2}}(\partial B)$ and $0\leq s\leq t<1$ then the following estimate holds
\begin{equation}\label{H}
(1-s)[u]^2_{H^\frac{1+s}{2}(\partial B)}\leq c_N(1-t)[u]^2_{H^\frac{1+t}{2}(\partial B)}. 
\end{equation}
\end{proposition}
\begin{proof}
By \eqref{armonic} and using the estimate for $\lambda_k$ established in Proposition~\ref{prola} we get
\begin{eqnarray*}
(1-s)[u]^2_{H^\frac{1+s}{2}(\partial B)} &=& (1-s)\sum_{k=0}^\infty \sum_{i=0}^{d(k)}\lambda_k^s a_k^i(u)^2 =
(1-s)\sum_{k=0}^\infty \sum_{i=0}^{d(k)}\lambda_k^{s-t}\lambda_k^{t} a_k^i(u)^2\\
&\leq & (1-s) \lambda_1^{s-t}\sum_{k=0}^\infty \sum_{i=0}^{d(k)}\lambda_k^{t} a_k^i(u)^2\\
&\leq & (1-s) \lambda_1^{s}c_N (1-t)\sum_{k=0}^\infty \sum_{i=0}^{d(k)}\lambda_k^{t} a_k^i(u)^2\\
&=& (1-s)s(N-s)\frac{P_s(B)}{P(B)}c_N (1-t)[u]^2_{H^\frac{1+t}{2}(\partial B)} \\
&\leq& c_N(1-t)[u]^2_{H^\frac{1+t}{2}(\partial B)}.
\end{eqnarray*}
\end{proof}
\begin{remark}\rm
We note that the result established in the previous proposition remains true also in the case $t=1$. Indeed, since
$$
\lim_{t\to 1^-}(1-t)[u]^2_{H^\frac{1+t}{2}(\partial B)}=[u]^2_{H^1(\partial B)}
$$ 
as established in \cite[Cor. 2]{BBM01}, we can pass to the limit $t\to1^-$ in \eqref{H}.
\end{remark}


\section{Preliminary estimates on the energy functional}\label{section2}

In the following we shall consider parameters $s,t\in (0,1)$ 
satisfying 
\begin{equation}\label{deltazero}
t-s\ge \delta_0>0\,.
\end{equation}
All the constants in this work, unless differently specified, will depend only on $N$ and $\delta_0$,
so that it will be possible to pass to the limits in a straightforward way
as $s\to 0^+$ or $t\to 1^-$.

\begin{proposition}\label{CC}
There exists $c_0=c_0(N,\delta_0)$ such that,
for any $E\subset\R^N$ and $0<s<t<1$ satisfying \eqref{deltazero},
it holds
\begin{equation}\label{stima_sotto_F}
 \F(E)\ge \frac{(1-t)P_t(E)}{2}-c_0 |E|.
 \end{equation}
\end{proposition}

\begin{proof} 
Set $m:=|E|$. We apply Young inequality with exponents~$\frac{t}{t-s}$
and~$\frac{t}{s}$ to the right hand side of \eqref{basicestimates} getting
\begin{eqnarray*}
c_N\frac 1 s\left(1-\frac st\right)^{-1}|E|^{1-\frac st}(1-t)^{\frac st} P_t(E)^{\frac st}&=& \left[ c_N\frac
{2^{\frac st}}{ s}\left(1-\frac st\right)^{-1} m^{1-\frac st}\right] \big[ 2^{-1} (1-t)P_t(E)\big]^{\frac st}\\
&\le& \left[ c_N\frac{2^{\frac st}}{ s}\left(1-\frac st\right)^{-1} m^{1-\frac
st}\right]^{\frac{t}{t-s}}+\frac{(1-t)P_t(E)}{2}.
\end{eqnarray*}
Thus~\eqref{basicestimates} gives that 
\begin{eqnarray*}
\F(E)&=&
 (1-t)P_t(E)-sP_s(E)\ge (1-t)P_t(E)-c_N\left( 1-\frac st\right)^{-1}|E|^{1-\frac s t}(1-t)^\frac s tP_t(E)^\frac s
t 
\\ 
&\ge& (1-t)P_t (E) -
\left[2^{\frac st} c_N\left(1-\frac st\right)^{-1} m^{1-\frac st}\right]^{\frac{t}{t-s}}-\frac{(1-t)P_t(E)}{2}
\\
&=& \frac{(1-t)P_t(E)}{2} - 
\left[2^{\frac st} c_N\frac{t}{t-s} \right]^{\frac{t}{t-s}}
m
\end{eqnarray*}
and this concludes the proof.
\end{proof}


\begin{corollary}\label{CC2}
Let $|E|=m$. Then both~$P_t(E)$ and~$P_s(E)$
are bounded above by quantities only depending on~$m$ and~$\F(E)$.
More explicitly
\begin{eqnarray}
\label{Y1} && (1-t)P_t(E)\le 2(\F(E)+c_0 m)\\
\label{Y2} {\mbox{and }}&& sP_s(E) \le c_0^{1-\frac st}m^{1-\frac s t} (\F(E)+c_0m)^{\frac s t},\end{eqnarray}
with~$c_0$ as in Proposition~\ref{CC}.
\end{corollary}

\begin{proof} We obtain~\eqref{Y1} easily from Proposition~\ref{CC}.
Then~\eqref{Y2} follows from~\eqref{basicestimates} 
and~\eqref{Y1}.\end{proof}

Now we define the isovolumetric function $\phi:(0,+\infty)\to \R$ as
$$
\phi(m)=\inf_{|E|=m}\F (E)\qquad m\in (0,+\infty).
$$
A general estimate on~$\phi(m)$ goes as follows:

\begin{lemma}\label{lema} We have
\begin{equation}\label{X1}
-c_0 m\le\phi(m)\le c_1 m^{\frac{N-t}{N}} \left(1 -\frac{c_2}{c_1} m^{\frac{t-s}N}\right),
\end{equation}
with~$c_0$ as in Proposition~\ref{CC} and 
\begin{equation}\label{c1c2}
c_1:=\frac{(1-t)P_t(B)}{|B|^{\frac{N-t}{N}}}\quad\text{and}\quad c_2:=\frac{s
P_s(B)}{|B|^{\frac{N-s}{N}}}.
\end{equation}
\end{lemma}

\begin{proof}
Let us begin by proving the estimate from above of $\phi(m)$.
For this, we take the unit ball~$B$
we set~$\rho:=(m/|B|)^{1/N}$ and we consider the ball~$B(0,\rho)$ of radius~$\rho$.
Notice that~$|B(0,\rho)|=\rho^N |B|=m$,
$$ 
P_t(B(0,\rho))=\rho^{N-t} P_t(B)
=\frac{P_t(B)}{|B|^{\frac{N-t}{N}}}\, m^{\frac{N-t}N}
$$
and
$$
 P_s(B(0,\rho))=\rho^{N-s} P_s(B)
=\frac{P_s(B)}{|B|^{\frac{N-s}{N}}}\,m^{\frac{N-s}N}
.
$$
By minimality, we get, with $c_1$ and $c_2$ as in \eqref{c1c2},
 \[
  \phi(m)\le \F(B(0,\rho))=(1-t)P_t(B(0,\rho))-sP_s(B(0,\rho))
=c_1 m^{\frac{N-t}N} \left(1 -\frac{c_2}{c_1} m^{\frac{t-s}N}\right),
 \]
 that proves~\eqref{X1}.

The first inequality in~\eqref{X1} follows from Proposition~\ref{CC}.
\end{proof}

\begin{remark}
{\rm We recall the fractional isoperimetric inequality, which holds true for any measurable set $E$ such that
$|E|<+\infty$: 
\begin{equation}\label{iso_ineq2}
|E|^{\frac{N-t}{N}} \leq c_N t(1-t)P_t(E).
\end{equation}
For the optimal constant $c_N$ we refer to \cite{FS08} (see in particular Equations~(1.10) and (4.2) there).}
\end{remark}
\begin{lemma}\label{positivenegative}
 There exist $m_0=m_0(N,\delta_0)$ and $m_1=m_1(N,\delta_0)$ such that:
\begin{eqnarray}
\label{E1}&&{\mbox{if $m>m_1$, then $\phi(m)<0$;}}\\ 
\label{E2}&&{\mbox{if $m\in(0,m_0)$, then $\phi(m)\ge \frac{c_N}{t}
m^{\frac{N-t}N}>0$.}} 
\end{eqnarray}
Moreover 
\begin{equation}\label{E3}
 \lim_{m\to 0^+}\phi(m)=0.
\end{equation}
\end{lemma}
\begin{proof}
We have that \eqref{E1} and~\eqref{E3} plainly follow from~\eqref{X1}.

Now we prove~\eqref{E2}. For this, we use Proposition~\ref{CC} and the fractional isoperimetric inequality in the
form \eqref{iso_ineq2} to obtain that, if~$|E|=m$,
$$ \F(E)\ge \frac{(1-t)P_t(E)}{2}-c_0 m \ge \frac{m^{\frac{N-t}N}}{2 \,c_N\, t}
-c_0 m = \frac{m^{\frac{N-t}N}}{2\,c_N\, t} \left( 1-2c_0\,c_N\, t\, m^{\frac t N}
\right).$$
In particular, if~$m$ is small enough, we have that
$$\F(E)\ge \frac{m^{\frac{N-t}N}}{4\,c_N\, t}$$
and this implies~\eqref{E2}.
\end{proof}


\begin{lemma}\label{lemstim}
Let $m_1$ be as in Lemma \ref{positivenegative}, and let $F$ 
be a minimizer of $\F$ among sets of measure $m>m_1$. We have
 \begin{equation}\label{disuguaglianze}
  \frac{c_N}{t}\, m^{\frac{N-t}{N}}\le (1-t)P_t(F)< c_0\, m\quad and \quad
\frac{c_N}{t}\,m^{\frac{N-t}{N}}< s P_s(F)\le c_0\,  m,
 \end{equation}
 for some $c_0>0$.
\end{lemma}

\begin{proof}
By Lemma \ref{positivenegative} we know that $(1-t)P_t(F)<sP_s(F)$, hence  
from \eqref{basicestimates} and from the fractional isoperimetric inequality \eqref{iso_ineq2} we get
\[
\frac{m^\frac{N-t}{N}}{c_N\,t}\le (1-t)P_t(F)< sP_s(F) \le c_0^{\frac{t-s}{t}} 2^{-\frac st }m^{1-\frac
st}[(1-t)P_t(F)]^{\frac st}
\]
with $c_0$ given in Proposition \ref{CC}. 
Then $(1-t)P_t(F)<c_0\, 2^{-\frac{s}{t-s}}\, m$.
This and~\eqref{basicestimates}
also implies the desired bound on~$sP_s(F)$.
\end{proof}

\begin{remark}\rm
By inspecting the proof of the Lemma \ref{positivenegative} we obtain explicit
estimates for $m_0$ and $m_1$:
\begin{eqnarray*}
m_0 &\geq& \left[4c_0 \,c_N\, t\right]^{-\frac N t}
\ =\ \left[4\left(\frac{c_N\, t\, 2^{s/t}}{t-s}\right)^{\frac{t}{t-s}} \,c_N\, t\right]^{-\frac N t}
\\
\\
m_1 &\le& \left(\frac{c_1}{c_2}\right)^{\frac{N}{t-s}}\ =\ 
\left[\frac{(1-t)P_t(B)}{sP_s(B)}\right]^{\frac{N}{t-s}} |B|.
\end{eqnarray*}
Moreover, the first inequality in the second formula in \eqref{disuguaglianze}
 \[
  \frac{c_N}{t}|F|^{\frac{N-t}{N}}<sP_s(F)
 \]
 entails that
$|F|\to\infty$ as $t\to0$ (and thus $\delta_0\to0$). Indeed, letting $t=s+\delta_0$, and using the fact that
$sP_s(F)\to N\omega_N|F|$ as $s\to0$, after an elementary computation we get that
\[
m_1\ge |F|\ge \left(\frac{c_N}{s+\delta_0}\right)^{\frac{N}{s+\delta_0}}.
\]
which gives also a lower bound on $m_1$ in terms of $s$ and $\delta_0$. Notice that if $t\to0$, then also $s$
and $\delta_0$ converge to $0$ and so  $m_1\to\infty$. Also it is not a direct consequence of our investigation, we
stress that it is natural to expect that also if only $s$ converges to $0$, then $m_1$ diverges to $+\infty$. 
\end{remark}
We state an elementary numerical inequality which will be useful in the proof of the forthcoming Proposition
\ref{nonoptimality}.
\begin{lemma}
Let $\gamma>0$ and $\lambda=(1+\gamma)^{1/N}$. Then, for any~$a$, $b\ge0$, it holds
\begin{equation}\label{Lz1}
(\l^{N-t}-1)a -(\l^{N-s}-1)b\le \gamma\,(a-b).
\end{equation} 
\end{lemma}
\begin{proof}
To prove \eqref{Lz1}, we notice that
$$ \lim_{\gamma\rightarrow0} (N-s)(1+\gamma)^{\frac{t-s}{N}}-(N-t)= t-s>0,$$
hence we may take~$\gamma$ small enough, such that
\begin{equation}\label{gs}
(N-s)(1+\gamma)^{\frac{t-s}{N}}-(N-t)\ge\frac{t-s}{2}.
\end{equation}
So we write
$$ f(\gamma):=\Big((1+\gamma)^{\frac{N-t}N}-1\Big)a -\Big(
(1+\gamma)^{\frac{N-s}N}-1\Big)b
=(\l^{N-t}-1)a -(\l^{N-s}-1)b,$$
and we notice that~$f(0)=0$ and
\begin{eqnarray*}
f'(\gamma) &=& {\frac{N-t}N} (1+\gamma)^{-\frac{t}N} a -
{\frac{N-s}N} (1+\gamma)^{-\frac{s}N} b \\&=&
{\frac{N-t}N} (1+\gamma)^{-\frac{t}N} (a-b)
-\frac{b(1+\gamma)^{-\frac{t}N}}{N} \big[ (N-s)(1+\gamma)^{\frac{t-s}{N}}-(N-t)\big]
\\ &\le& (a-b) -
\frac{b(1+\gamma)^{-\frac{t}N} (t-s)}{2N},
\end{eqnarray*}
thanks to~\eqref{gs}. In particular, $f'(\gamma)\le (a-b)$
and thus~$f(\gamma)\le \gamma\,(a-b)$, that establishes~\eqref{Lz1}.
%

\end{proof}

\begin{proposition}[Non-optimality criterion]\label{nonoptimality}
There exists $\eps=\eps(N,\delta_0)$ such that, 
if $F\subset\R^N$ can be written as  
$F=F_1\cup F_2$, with $|F_1\cap F_2|=0$, 
\begin{eqnarray}
&& \label{hp0} |F_2|\le \eps\min(1,|F_1|),\\
{\mbox{and }}&&\label{hp1}
 (1-t)[P_t(F_1)+P_t(F_2)-P_t(F)]\le \frac{\F(F_2)}{2},
\end{eqnarray}
then there exists a set $G$ with $|G|=|F|$ and $\F(G)<\F(F)$
(i.e., $F$ is not a minimizer).

In addition, we have that the set~$G$
is either a ball of volume $m$,
or a dilation of the set~$F_1$,
according to the following formula:
\begin{equation}\label{G}
G\,=\, \sqrt[N]{1+\frac{|F_2|}{|F_1|}}\; F_1
\end{equation}
\end{proposition}

\begin{proof} 
Let~$m:=|F|$, $m_1:=|F_1|$ and $m_2:=|F_2|$. 
We may suppose that~$\F(F)$ is less than or equal than~$\F$
of the ball of volume~$m$, $B_m$, otherwise we can take~$G$ equal
to such a ball, decrease the energy and finish our proof. That is,
we may suppose that
 \begin{equation}\label{e1}
  \F(F)\le \F(B_m)\le (1-t)P_t(B_m)\leq \frac{(1-t)P_t(B)}{|B|^{\frac{N-t}{N}}} m^{\frac{N-t}{N}}. \end{equation}
Let $G=\l F_1$, with~$\l:=\sqrt[N]{1+\gamma}$ and~$\gamma=m_2/m_1$.
Notice that
this is in agreement with~\eqref{G}, and also~$|G|=m$.
Moreover, by~\eqref{hp0} we have that
$$ \gamma\le\frac{\eps\min(1,m_1)}{m_1}\le\eps,$$
so that $\gamma\in(0,1)$ can be taken as small as we like.

By applying inequality \eqref{Lz1} with $a=(1-t)P_t(F_1)$ and $b=sP_s(F_1)$, we obtain that
$$ 
(\l^{N-t}-1)\,(1-t)P_t(F_1) -(\l^{N-s}-1)\,sP_s(F_1)\le 
\gamma\,[(1-t)P_t(F_1)-sP_s(F_1)].
$$
As a consequence we get
\[
 \begin{aligned}
  \F(G)&= (1-t)P_t(G)-sP_s(G)\\
&=
\l^{N-t}(1-t)P_t(F_1)-\l^{N-s}sP_s(F_1)
\\ &=\F(F_1)+\big[(\l^{N-t}-1)(1-t)P_t(F_1)-(\l^{N-s}-1)sP_s(F_1)\big] \\
  &\le  (1+\gamma)\F(F_1)\,.
 \end{aligned}
\]
 Thus we have, by \eqref{R2} and~\eqref{hp1},
 \begin{eqnarray}\label{R3}
 \nonumber\F(G)-\F( F)&\le& (1+\gamma)\F(F_1)-(1-t)P_t(F)+sP_s(F)\\
  \nonumber  &\le& (1+\gamma)\F(F_1)-(1-t)P_t(F)+sP_s(F_1)+sP_s(F_2)\\
   \nonumber &\le& (1+\gamma)\F(F_1)+sP_s(F_1)+sP_s(F_2)\\
   \nonumber &&+\frac12\F(F_2)-(1-t)P_t(F_1)-(1-t)P_t(F_2)\\
   &=&\gamma\F(F_1)-\frac12\F(F_2). 
 \end{eqnarray}
Furthermore by~\eqref{E2}, since $m_2$ can be chosen in $(0,m_0)$, $m_0$ as in Lemma~\ref{positivenegative} (up to decreasing the value of $\eps$), we have
\begin{equation}\label{R4}
 \F(F_2)\ge\phi(m_2)\ge \frac{c_N}{t}\, m_2^\frac{N-t}{N}.
\end{equation}
Also, using again~\eqref{R2} and~\eqref{hp1}, we have that
\[
\begin{aligned}
 \F(F_1)&=[(1-t)P_t(F_1)+(1-t)P_t(F_2)-sP_s(F_1)-sP_s(F_2)]-\F(F_2)\\
 &\leq \F(F)+[(1-t)P_t(F_1)+(1-t)P_t(F_2)-(1-t)P_t(F)-\F(F_2)]\\
 &\le\F(F)-\frac12\F(F_2)<\F(F).
\end{aligned}
\]
This, \eqref{R3} and~\eqref{R4} give that
$$ \F(G)-\F( F)\le \gamma \F( F)-\frac12\F( F_2)\le
\gamma \F( F)-\frac{c_N}{2t}\, m_2^\frac{N-t}{N}.$$

Accordingly, recalling \eqref{e1} we conclude that
 \[
  \begin{aligned}
   \F(G)-\F( F)&\le \frac{(1-t)P_t(B)}{|B|^{\frac{N-t}{N}}}\, \gamma (m_1+m_2)^{\frac{N-t}N}-
\frac{c_N}{2t}\,m_2^{\frac{N-t}N}\\
&= m_2^{\frac{N-t}N}\left[ \frac{(1-t)P_t(B)}{|B|^{\frac{N-t}{N}}}\,\gamma \Big(
\gamma^{-1}+1\Big)^{\frac{N-t}N}-\frac{c_N}{2t}\right]
\\ &\le m_2^{\frac{N-t}N}\left[ 
\frac{(1-t)P_t(B)}{|B|^{\frac{N-t}{N}}}\,\gamma 
\left(2\gamma^{-1}\right)^{\frac{N-t}N}-\frac{c_N}{2t}\right]
\\ &= m_2^{\frac{N-t}N}\left[
\frac{2^{\frac{N-t}{N}}(1-t)P_t(B)}{|B|^{\frac{N-t}{N}}}\,\gamma^{\frac{t}N}-\frac{c_N}{2t}\right]
  \end{aligned}
 \]
which is negative
if~$\gamma$ is small enough, i.~e. 
$$
\gamma< \left[\frac{c_N}{2t}
\frac{|B|^{\frac{N-t}{N}}}{2^{\frac{N-t}{N}}(1-t)P_t(B)}\right]^{\frac Nt}.
$$
The proof is concluded.
\end{proof}

When~\eqref{hp1} does not hold, one obtains for free
some interesting density bounds.

Given a measurable set $E$ we denote by $\partial^m E$ the measure theoretic boundary of $E$
defined as 
\[
\partial^m E = \{x\in\R^N:\ |E\cap B_r(x)|>0\text{ and }|E\setminus B_r(x)|>0\text{ for all }r>0\}.
\]

\begin{lemma}\label{Aux}
Let~$F$ be a set of finite $t$-perimeter and volume~$m$,
and let~$x_0\in\R^N$. Assume
\begin{eqnarray}
\label{F1}
&&{\mbox{either $F_1:=F\setminus B(x_0,r)$ and~$F_2:=F\cap B(x_0,r)$,}}
\\ \label{F2}
&&{\mbox{or $F_2:=F\setminus B(x_0,r)$ and~$F_1:=F\cap B(x_0,r)$,}}
\end{eqnarray}
and suppose that $|F_2|<m_0$, with $m_0$ be as in Lemma~\ref{positivenegative},
and
\begin{equation}\label{no-hp1}
 (1-t)[P_t(F_1)+P_t(F_2)-P_t(F)]\ge \frac{\F(F_2)}{2}.
\end{equation}
Then
\begin{equation}\label{Au0}
\int_{F_1}\int_{F_2}\frac{dx\,dy}{|x-y|^{N+t}} \ge \frac{c_N}{t(1-t)} |F_2|^{\frac{N-t}N}.
\end{equation}
If~$x_0\in\partial^m F$ and \eqref{no-hp1} holds for any $r\le r_0$, 
we also have the estimate
\begin{equation}\label{Au1} |F\cap B(x_0,r)|\ge c_0\, r^N
\qquad \text{for all }r\in (0,r_0],
\end{equation}
where the constant $c_0>0$ depend only on $N$ and $\delta_0$.
\end{lemma}

\begin{proof} Without loss of generality we can assume $x_0=0$. Also, using either~\eqref{F1} or~\eqref{F2},~\eqref{no-hp1} and~\eqref{eqsum}, we have that
$$ \int_{F_1}\int_{F_2}\frac{dx\,dy}{|x-y|^{N+t}} =\frac{1-t}{2(1-t)}
(P_t(F_1)+P_t(F_2)-P_t(F))\ge\frac{\F(F_2)}{4(1-t)} \ge\frac{\phi(|F_2|)}{4(1-t)}.$$
This and~\eqref{E2} (which can be used here thanks to the fact that we are assuming $|F_2|<m_0$) imply~\eqref{Au0}.

Now we prove~\eqref{Au1}.
For this, we take~$F_1$ and~$F_2$ as in~\eqref{F1}
and we define $\mu(r):=~|B(0,r)\cap F|=|F_2|$. Note that by the co-area formula
$$ 
\mu'(r)=\mathcal H^{N-1}(\partial B(0,r)\cap F), \quad \text{for a. e. } r.
$$
Then, by \eqref{Au0} and the fact that $F_1:=F\setminus B(0,r)\subset (B(0,r))^c$,
$$
\frac{c_N}{t(1-t)}\, \mu(r)^{\frac{N-t}{N}}\leq \int_{F_2}\int_{F_1}\frac{dx\,dy}{|x-y|^{N+t}} \leq \int_{F_2}
\int_{(B(0,r))^c} \frac{dx\,dy}{|x-y|^{N+t}}.
$$
For any $x\in F\cap B(0,r)$, we have
$$
\int_{(B(0,r))^c} \frac{dy}{|x-y|^{N+t}}\leq \int_{(B(x,r-|x|))^c}\frac{dy}{|x-y|^{N+t}}= \frac{N\omega_N}{t}
(r-|x|)^{-t}
$$
that leads to
$$
\int_{F_2} \int_{(B(0,r))^c} \frac{dx\,dy}{|x-y|^{N+t}}\leq \frac{c_N}{t}\int_0^r \mu'(z)(r-z)^{-t}\,dz.
$$
Finally we arrive at the following integro-differential inequality
$$
\mu(r)^{\frac{N-t}{N}}\leq c_N(1-t) \int_0^r \mu'(z)(r-z)^{-t}\,dz.
$$
We may integrate the last inequality in the $r$ variable on the interval $(0,\rho)$ and get
$$
\int_{0}^{\rho} \mu(r)^{\frac{N-t}{N}}\, dr\leq c_N(1-t) \int_{0}^{\rho}\int_0^r \mu'(z)(r-z)^{-t}\,dz\,dr,
$$
interchanging the order of integration,
$$
\int_{0}^{\rho}\int_0^r \mu'(z)(r-z)^{-t}\,dz\,dr=\int_0^{\rho} \mu'(z)\int_{z}^{\rho}(r-z)^{-t} \, dr\,dz,
$$
we get 
$$
\int_{0}^{\rho} \mu(r)^{\frac{N-t}{N}}\, dr\leq c_N\, \rho^{1-t} \mu (\rho).
$$
Now we arrive at the desired result, indeed, following \cite{CG11} (see the end of p. 9), it is possible to prove
that 
\begin{equation}\label{eq_int_diff} 
\mu(r)\geq g(r):= \left[\frac{1}{2c_N(N+1-t)}\right]^{\frac Nt }
\, r^N
\end{equation}
for any $r< r_0=(m_0/\omega_N)^{1/N}$, where $g$ satisfies 
$$
\int_{0}^{\rho} g(r)^{\frac{N-t}{N}}\, dr\geq 2 c_N\, \rho^{1-t} g (\rho),
$$
with the same constant $c_N$ as in \eqref{eq_int_diff}.  
\end{proof}

The combination of
Proposition~\ref{nonoptimality} and Lemma~\ref{Aux}
yield the following density estimate:

\begin{proposition}\label{densityestimate}
There exist $r_0=r_0(m,N,\delta_0)>0$ such that, 
if $F$ is a minimizer for $\phi(m)$ and 
$x_0\in\partial^m F$, there holds
 \[
|B(x_0,r)\cap F|\ge c_0\, r^N
 \]
for any $r<r_0$, where $c_0$ is as in \eqref{Au1}.
\end{proposition}

\begin{proof}
Let $F_1$ and~$F_2$ be as in~\eqref{F1}.
Up to choosing $r_0$ 
small enough, that is,
$$
\omega_N r_0^N \le \eps(N,\delta_0)\min(1,m)\,,
$$
we can suppose that $F_1$ and $F_2$ satisfy the hypotheses 
of Proposition \ref{nonoptimality}. Thus, since $F$ is a minimum, we 
obtain that~\eqref{hp1} cannot hold true. Hence~\eqref{no-hp1}
is satisfied, and so we can apply~\eqref{Au1} in Lemma~\ref{Aux}
and obtain the desired result.
\end{proof}

\begin{proposition}\label{necessario}
Let $F$ be a minimum for $\phi(m)$. Then $F$ is essentially bounded. 
Moreover, if $t=1$, for any $s<t$, $s\in (0,1)$, $F$ 
is also essentially connected in the sense of \cite{AmbCasMasMor},
that is, it cannot be decomposed into two 
disjoint sets $F_1$ and $F_2$ of positive measure such that
$P(F)=P(F_1)+P(F_2)$.
\end{proposition}

\begin{proof}

\noindent Let~$F$ be a minimum. First we prove that it is bounded.
By contradiction, if not, there exists a sequence $x_k\in\partial^m F$ 
such that $|x_k|\to\infty$ as $k\to\infty$. 
In particular, up to a subsequence, we may suppose that all the balls~$B(x_k,1)$
are disjoint, hence so are the balls~$B(x_k,r)$ when~$r\in(0,1)$.
Hence
$$ m=|F|\ge \sum_k|B(x_k,r)\cap F|.$$
On the other hand, by Proposition \ref{densityestimate}, we 
know that~$|B(x_k,r)\cap F|\ge c_0 r^N$ if~$r$ is small enough, hence
we obtain that
\[
m\ge
\sum_k c_0 r^N=+\infty,
 \]
which is clearly not possible.

This proves that~$F$ is bounded.
Now we show that, if $t=1$, $F$ is also
essentially connected.
Suppose, by contradiction, that~$F$ can be decomposed into two
disjoint sets $F_1$ and $F_2$ of positive measure such that
\begin{equation}\label{No}
P(F)=P(F_1)+P(F_2).\end{equation} Since $F$ is bounded, so are~$F_1$ and~$F_2$,
say~$F_1$, $F_2\subseteq B(0,R)$, for some~$R>0$.
Hence, we consider the translation~$F_{2,k}:=F_2+(k,0,\dots,0)$
and we observe that if~$x\in F_1$ and~$y\in F_{2,k}$
we have that
$$ |x-y|\ge |y|-|x|\ge k-2R\ge \frac{k}{2}$$
if~$k$ is large enough. Accordingly, we have that
$$ \int_{F_1}\int_{F_{2,k}}\frac{dx\,dy}{|x-y|^{N+s}}\le
\int_{B(0,R)}\int_{B(0,R)+(k,0,\dots,0)}\frac{dx\,dy}{(k/2)^{N+s}}
= \frac{c_N R^{2N}}{k^{N+s}}$$
and so
\begin{equation}\label{k}
\lim_{k\rightarrow+\infty} \int_{F_1}\int_{F_{2,k}}\frac{dx\,dy}{|x-y|^{N+s}}=0.
\end{equation}
Notice also that, if~$G_k:=F_1\cup F_{2,k}$ we have that~$|G_k|=
|F_1|+|F_{2,k}|=|F_1|+|F_2|=|F|$, for~$k$ large, and so,
by the minimality of~$F$, \eqref{eqsum}, \eqref{R2} and \eqref{No} we have that
\begin{eqnarray*}&&
N\omega_N P(F_1)+N\omega_NP(F_2)
-sP_s(F_1)-sP_s(F_2)+
2s\int_{F_1}\int_{F_2}\frac{dx\,dy}{|x-y|^{N+s}} \\&=&
N\omega_N P(F)-sP_s(F)\\ &=&
\mathcal{F}_{s,1}(F) \\&\le& \mathcal{F}_{s,1}(G_k)
\\ &=& N\omega_N P(G_k)-sP_s(G_k)
\\ &\le& N\omega_NP(F_1)+N\omega_NP(F_{2,k}) 
-sP_s(F_1)-sP_s(F_{2,k})+
2s\int_{F_1}\int_{F_{2,k}}\frac{dx\,dy}{|x-y|^{N+s}}
\\ &=& N\omega_NP(F_1)+N\omega_NP(F_{2})  
-sP_s(F_1)-sP_s(F_{2})+
2s\int_{F_1}\int_{F_{2,k}}\frac{dx\,dy}{|x-y|^{N+s}}.
\end{eqnarray*}
Therefore, taking the limit as~$k\rightarrow+\infty$ and using~\eqref{k},
we obtain that
$$ 2s\int_{F_1}\int_{F_2}\frac{dx\,dy}{|x-y|^{N+s}}\le0.$$
This says that either~$F_1$ or~$F_2$ must have zero measure,
against our assumptions.
\end{proof}

We conclude the section with the following estimate on 
the fractional isoperimetric deficit, which 
will be important to localize minimizing sequences.

\begin{lemma}\label{deficit}
There exists $m_2=m_2(N,\delta_0)$ such that for any $m\in(0,m_2)$ the following
statement holds true.

Let~$F\subset \mathbb{R}^N$ be a set of finite
perimeter. Assume that~$\F(F)\le\F(B_m)$.
Then there exists $c_0>0$ such that 
\begin{equation}\label{deficit1}
D_t(F):=\frac{P_t(F)-P_t(B_m)}{P_t(B_m)}\leq c_0\,m^{\frac{t-s}N}\,.
\end{equation}
In addition, there exists a translation of~$F$ (still denoted
by~$F$ for simplicity) such that
\begin{equation}
\label{deficit2}
|F\,\Delta \,B_m|\leq c_0 \, m^{1+\frac{t-s}{2N}}\,.
\end{equation}
\end{lemma}

\begin{proof}
First recall that
\begin{equation}\label{PtB}
P_t(B_m)=\frac{P_t(B)}{|B|^{\frac{N-t}{N}}} \, m^{\frac{N-t}{N}}.
\end{equation}
Also, by our assumptions,
\begin{equation}\label{2E}
(1-t)P_t(F)-sP_s(F)=\F(F)\le \F(B_m)\leq (1-t)P_t(B_m).
\end{equation}
Using \eqref{Y2} we have that
 \begin{eqnarray*}
 s P_s(F) &\le& c_0^{1-\frac st}m^{1-\frac s t} [(1-t)P_t(B_m)+c_0 \,m]^{\frac st}\\
&= &c_0^{1-\frac st}m^{1-\frac s t}
\left[\frac{(1-t)P_t(B)}{|B|^{\frac{N-t}{N}}}\,m^{\frac{N-t}{N}}+c_0\, m\right]^{\frac st}\\
& \le& c_0^{1-\frac st} \left[\frac{(1-t)P_t(B)}{|B|^{\frac{N-t}{N}}}+c_0\right]^{\frac
st}m^{\frac{N-s}{N}},\end{eqnarray*}
for small~$m$.
{F}rom this and~\eqref{2E},
we have that
$$ \frac{P_t(F)-P_t(B_m)}{P_t (B_m)} \le c_0^{1-\frac s
t}\left[\frac{(1-t)P_t(B)}{|B|^{\frac{N-t}{N}}}+c_0\right]^{\frac st}
\frac{|B|^{\frac{N-t}{N}}}{(1-t)P_t(B)} \, m^{\frac{N-s}{N}-\frac{N-t}{N}}\le c_0 m^\frac{t-s}{N}\,.
$$
This proves~\eqref{deficit1}. 

To prove~\eqref{deficit2} it is sufficient to use \eqref{deficit1} 
and the estimate 
$$
c_0\,D_t(F)\geq \frac{|F\,\Delta \,B_m|^2}{|B_m|^2},
$$ 
which was proved in \cite[Theorem 1.1]{labandadei5} for any $t\geq \delta_0>0$.
Together with \eqref{deficit1} and possibly increasing the constant $c_0$,
this implies \eqref{deficit2}.
\end{proof}

\section{Existence of minimizers}\label{section3}

In order to prove the first statement in Theorem \ref{thexist},
and for further use as well, we prove a general result
on integro-differential equations:

\begin{lemma}\label{ID}
Let~$m,t\in(0,1)$.
Let~$c,\bar\rho\ge0$ be such that 
\begin{equation}\label{condc}
c\ge (1-t)m^\frac tN\,,
\end{equation}
and~let $\mu:[0,+\infty)\rightarrow[0,m]$
be a non-increasing
function such that 
\begin{equation}\label{U}
-\int_{\rho}^\infty \mu'(z)(z-\rho)^{-t}\, dz\geq \frac{3c}{1-t} \mu(\rho)^{\frac{N-t}{N}}\qquad \text{for all }\rho\ge\bar\rho\,.
\end{equation}
Then, there holds
\begin{equation}\label{eqlu}
\mu\left(\bar\rho+\frac{(2m)^{\frac t N}N}{ct}\right)=0\,.
\end{equation}
\end{lemma}

\begin{proof}
Integrating \eqref{U} between $R\ge\bar\rho$ and $+\infty$, we obtain
\begin{equation}\label{UU}
-\int_R^\infty\left(\int_\rho^\infty \mu'(z)(z-\rho)^{-t}\, dz\right)\,d\rho
\geq \frac{3c}{1-t} \int_R^\infty
\mu(\rho)^{\frac{N-t}{N}}\, d\rho.
\end{equation}
Also, if~$z\in[R,R+1]$ we have that~$z-R\le1$ and so, since~$\mu'\le0$
a. e., we get that
$$ -\int_R^{R+1} \mu'(z)(z-R)^{1-t}\, dz\le -\int_R^{R+1} \mu'(z)\, dz
=\mu(R)-\mu(R+1).$$
Therefore,
interchanging the order of integration in~\eqref{UU}, 
integrating by parts
and using that $\mu\in[0,m]$ and \eqref{condc}, we see that
\begin{eqnarray*}
-\int_R^\infty\left(\int_\rho^\infty \mu'(z)(z-\rho)^{-t}\, dz\right)\,d\rho &=&
-\int_R^\infty\left(\int_{R}^{z} \mu'(z)(z-\rho)^{-t}\,d\rho\right)\,dz
\\
&= &-\frac{1}{1-t}\int_R^\infty \mu'(z)(z-R)^{1-t} \, dz 
\\
&\leq& \frac{\mu(R)-\mu(R+1)}{1-t}
-\frac{1}{1-t}\int_{R+1}^\infty \mu'(z)(z-R)^{1-t} \, dz
\\
&=& \frac{\mu(R)}{1-t}+\int_{R+1}^\infty \mu(z)(z-R)^{-t} \, dz
\\
&\le& \frac{\mu(R)}{1-t}+\int_{R+1}^\infty \mu(z) \, dz
\\
&\le& \frac{\mu(R)}{1-t}+m^\frac t N\int_{R}^\infty \mu(z)^\frac{N-t}{N} \, dz
\\
&\le& \frac{1}{1-t}
\left( \mu(R)+c\int_{R}^\infty \mu(z)^\frac{N-t}{N} \, dz\right).
\end{eqnarray*}
Recalling \eqref{UU},
this gives the integro-differential inequality
\begin{equation}\label{usecondo}
\mu(\rho) \ge 2c
\int_{\rho}^\infty \mu(z)^\frac{N-t}{N}dz 
\qquad \text{for all }\rho\ge \bar\rho\,.
\end{equation}
Let now 
$$
g(\rho):=\left\{\begin{array}{ll}
\left[(2\mu(\bar\rho))^{\frac t N}-\dfrac{ct}{N}(\rho-\bar\rho)\right]^\frac{N}{t} 
& \text{if }\rho\in \left[\bar \rho,\bar\rho+\dfrac{(2\mu(\bar\rho))^{\frac tN}N}{ct}\right]
\\ 
\\
0 & \text{if }\rho>\bar\rho+\dfrac{(2\mu(\bar\rho))^{\frac tN}N}{ct}\,.
\end{array}\right.
$$
Notice that $g$ is continuous and it satisfies
\begin{equation}\label{4.5bis}
2c\int_\rho^\infty g(z)^\frac{N-t}{N} \,dz =
2g(\rho) \qquad \text{for all }
\rho\in \left[\bar \rho,\bar\rho+\dfrac{(2\mu(\bar\rho))^{\frac tN}N}{ct}\right].\end{equation}
We now claim that 
\begin{equation}\label{claimb}
g(\rho)\ge \mu(\rho)
\qquad \text{for all }
\rho\in \left[\bar \rho,\bar\rho+\dfrac{(2\mu(\bar\rho))^{\frac tN}N}{ct}\right].
\end{equation}
Indeed, we consider the set~$I:=\{\rho>\bar\rho:\ \mu(z)\ge g(z) \text{ for all }z\ge \rho\}$.
By construction, $I\subseteq [\bar\rho,+\infty)$.
Furthermore, if~$z\ge \bar\rho+[(2\mu(\bar\rho))^{\frac tN}N]/ct$
then~$g(z)=0\le\mu(z)$, therefore~$\bar\rho+[(2\mu(\bar\rho))^{\frac tN}N]/ct\in I$.
As a consequence, we can define~$R_*:= \inf I$,
and we have that
\begin{equation}\label{R*}
R_*\in [\bar\rho,\;
\bar\rho+[(2\mu(\bar\rho))^{\frac tN}N]/ct].\end{equation}
By definition of $R_*$, there exists a 
sequence $R_n\to R_*$, with $R_n\le R_*$, such that $g(R_n)>\mu(R_n)$.
Then, recalling~\eqref{usecondo} and~\eqref{4.5bis}, we have
\begin{equation}\begin{split}\label{leppa}
g(R_n)\,&>\mu(R_n)\\ &\ge2c\int_{R_n}^\infty \mu(z)^\frac{N-t}{N}dz
\\
&\ge 2c\int_{R_n}^{R_*} \mu(z)^\frac{N-t}{N}dz +
2c\int_{R_*}^\infty g(z)^\frac{N-t}{N}dz
\\
&= 2c\int_{R_n}^{R_*} \mu(z)^\frac{N-t}{N}dz + 2g(R_*).
\end{split}\end{equation}
Passing to the limit in \eqref{leppa} as $n\to +\infty$ we get
$g(R_*)\ge 2g(R_*)$, which means~$g(R_*)=0$. This implies
that~$R_*\ge 
\bar\rho+[(2\mu(\bar\rho))^{\frac tN}N]/ct$.

This information, combined with~\eqref{R*}, gives that~$R_*=
\bar\rho+[(2\mu(\bar\rho))^{\frac tN}N]/ct$, and this in turn
implies~\eqref{claimb}.

Then, we evaluate~\eqref{claimb} at~$\rho=\bar\rho+[(2\mu(\bar\rho))^{\frac tN}N]/ct$
and we obtain~\eqref{eqlu}.

\end{proof}

With the above result, we are able 
to prove the first statement in Theorem \ref{thexist},
concerning the existence of minimizers for small volumes. 

\begin{theorem}\label{PRE-thexist}
For any $0\leq s<t\leq 1$, $t-s\geq \delta_0>0$, there exists
$\bar m_0=\bar m_0(N,\delta_0)>0$ such that for all $m\in(0,\bar m_0)$, problem
\eqref{problem} has a minimizer $F\subset \mathbb{R}^N$.
\end{theorem}

\begin{proof}

Suppose $0<s<t<1$. We use the Direct Method of the Calculus of Variations. 
Let us consider a minimizing 
sequence $\{F_k\} \subset \mathbb{R}^N$, that is a sequence of sets of 
finite $t$-perimeter $F_k$ with $|F_k|=m$ such that 
\begin{equation}\label{78}
\lim_{k\to \infty} 
\F(F_k) =\phi(m). \end{equation}
Let also set $r_m:=(m/\omega_N)^{1/N}>0$, so that $|B(0,r_m)|=m$. 
Our goal is to show that
we can reduce ourselves to the case in which~$F_k$
lies in a large ball, independent of~$k$. More precisely,
we claim that there exist~$\rho_*>0$ and sets~$G_k$, with~$|G_k|=m$, such that
\begin{equation}\label{MI}
{\mbox{$G_k\subseteq B(0,\rho_*)$ \ and \ $\F(G_k)\le\F(F_k)$.}}
\end{equation}
To prove it, 
we take~$\rho\ge r_m$ and we set 
\begin{equation}\label{DD}
{\mbox{$X_k^\rho:=F_k\cap B(0,\rho)$ 
and $Y_k^\rho:=F_k\setminus B(0,\rho)$.}}\end{equation}
We distinguish two cases:
\begin{eqnarray}
{\mbox{either for any $\rho\ge r_m$ we have}}&&\label{case 1}
(1-t)[P_t(X_k^\rho)+P_t(Y_k^\rho)-P_t(F_k)]\ge
\frac{\F(Y_k^\rho)}{2}
\\ {\mbox{or there exists $\rho\ge r_m$ such that}}&&\label{case 2}
(1-t)[P_t(X_k^\rho)+P_t(Y_k^\rho)-P_t(F_k)]\leq
\frac{\F(Y_k^\rho)}{2}.
\end{eqnarray}
Let us first deal with~\eqref{case 1}.
In this case
we can apply Lemma~\ref{Aux} using the setting 
in~\eqref{F2}: accordingly, from~\eqref{Au0} we see that 
$$ 
\int_{X_k^{\rho}}\int_{Y_k^{\rho}}\frac{dx\,dy}{|x-y|^{N+t}} \ge  \frac{c_N}{t(1-t)}  |Y_k^{\rho}|^{\frac{N-t}N}
.$$ 
Let us define the non-increasing function 
$\eta(\rho):=|F_k\setminus B(0,\rho)|=|Y^\rho_k|$. Note that by the co-area formula
$$ 
\eta'(\rho)=-\mathcal H^{N-1}(\partial B(0,\rho)\cap F), \quad \text{for a. e. } \rho>0.
$$

Proceeding as in the proof of Lemma  \ref{Aux}, we have
\begin{eqnarray*}
\int_{Y_k^{\rho}}\int_{X_k^{\rho}}\frac{dx\,dy}{|x-y|^{N+t}} &\leq&
\int_{Y_k^{\rho}}\int_{B(0,\rho)}\frac{dx\,dy}{|x-y|^{N+t}}
\\
&\leq& 
\int_{Y_k^{\rho}}\left(\int_{(B(y,|y|-\rho))^c}\frac{dx}{|x-y|^{N+t}}\right)dy
\\
&\leq& -\frac{N\omega_N}{t} \int_{\rho}^{\infty}
\eta'(z)(z-\rho)^{-t}\, dz,
\end{eqnarray*}
whence
$$ -\int_{\rho}^\infty \eta'(z)(z-\rho)^{-t}\, dz\geq 
\frac{c_N}{1-t} \eta(\rho)^{\frac{N-t}{N}},$$
that is, $\eta$ satisfies inequality~\eqref{U}.
We now apply Lemma~\ref{ID} with~$\mu=\eta$, $c=c_N/3$
and~$\bar\rho=r_m$. Notice that, possibly reducing $\bar m_0$,
we can ensure that condition \eqref{condc} is satisfied.
{}From \eqref{eqlu} we conclude that
$$
\eta\left(r_m+\frac{3(2m)^{\frac tN}N}{c_N t}\right)=0\,, 
$$
that is,
\[
F_k\subseteq B\left(0,r_m+\frac{3(2m)^{\frac tN}N}{c_Nt}\right).
\]
\noindent
This proves~\eqref{MI} with $\rho_*$ given by
$$
\rho_*:=r_m+\frac{3(2m)^{\frac t N}N}{c_Nt}
$$ 
in the case where~\eqref{case 1} holds (here one can take~$G_k:=F_k$).

We now deal with case~\eqref{case 2}.
In this case, 
we use \eqref{deficit2}
and we obtain (up to a translation of~$F_k$ that is still denoted
by~$F_k$) that 
$$|F_k\setminus B(0,r_m)|+|B(0,r_m)\setminus F_k|=
|F_k\,\Delta \,B(0,r_m)|\leq c_0 \,m^{1+\frac{t-s}{2N}},$$
$c_0$ as in Lemma~\ref{deficit}.
In particular, if~$\rho\ge r_m$ is the one given by~\eqref{case 2}
we have that
\begin{eqnarray*}
|F_k\cap B(0,\rho)| &\ge& |F_k\cap B(0,r_m)|
\\ &=& |B(0,r_m)| -|B(0,r_m)\setminus F_k|
\\ &\ge& m-c_0\, m^{1+\frac{t-s}{2N}}
\\ &\ge&\frac{m}{2}
\end{eqnarray*}
if~$m$ is small enough, i.~e.
\begin{equation}\label{stima_m_0}
m\leq \left[\frac{1}{2c_0}\right]^{\frac{2N}{t-s}}
\end{equation}
and moreover
$$ |F_k\setminus B(0,\rho)|\le |F_k\setminus B(0,r_m)|
\leq c_0\,m^{1+\frac{t-s}{2N}}.$$
Therefore, for small~$m$, recalling~\eqref{DD} we see that
\begin{equation*}
\begin{split}
& 2c_0\, m^{\frac{t-s}{2N}}\,\min\big(1,|X_k^\rho|\big)
=
 2c_0\, m^{\frac{t-s}{2N}}\,\min\big(1,|F_k\cap B(0,\rho)|\big)\ge 
 2c_0\, m^{\frac{t-s}{2N}}\,
\frac{m}{2}\\
&\qquad=c_0\, m^{1+\frac{t-s}{2N}}\ge |F_k\setminus B(0,\rho)|=|Y_k^\rho|.
\end{split}\end{equation*}
Thanks to this and~\eqref{case 2},
we can apply Proposition~\ref{nonoptimality},
with~$\eps:=2c_0\, m^{\frac{t-s}{2N}}$,
$F_1:=X_k^\rho$ and $F_2:=Y_k^\rho$.

Hence, from Proposition~\ref{nonoptimality},
we find~$G_k$ such that~$\F(G_k)\le\F(F_k)$;
notice also that, in light of~\eqref{G},
we know that~$G_k$ is either a ball or a dilation of~$X_k^\rho$,
which is contained in~$B(0,2\rho)$. Thus also $G_k$
is contained in a ball of universal radius, and this
establishes~\eqref{MI} also in case~\eqref{case 2}.

Thus, by~\eqref{MI},
we have constructed a minimizing sequence~$G_k$
that is uniformly contained in a fixed ball. 
By Proposition~\ref{CC}, we also obtain that
$$ (1-t)P_t(G_k)\le 2[\F(G_k)+c_0 m]\le 2[\F(B(0,r_m))+c_0 m],$$
hence the $t$-perimeter of~$G_k$ is bounded uniformly in~$k$.

By the compact embedding of $H^\frac{t}{2}$ into $H^\frac{s}{2}$ 
(see \cite[Section 7]{DPV12}), 
up to extracting a subsequence, the sets $G_k$ converge in $W^{s,1}$ 
(hence also in $L^1$) to a limit set $G$, and it holds
$$ \lim_{k\to+\infty} P_s(G_k)=P_s(G).$$
The lower semicontinuity of the $t$-perimeter yields that
$$ \liminf_{k\to+\infty} P_t(G_k)\ge P_t(G)$$
Hence, by~\eqref{78} and~\eqref{MI},
\begin{eqnarray*} 
\F(G)&=&(1-t)P_t(G)-sP_s(G)\le
\liminf_{k\to+\infty}[(1-t) P_t(G_k)-sP_s(G_k)]\\
&=&
\liminf_{k\to+\infty} \F(G_k)
\le\liminf_{k\to+\infty} \F(F_k)\le\phi(m),
\end{eqnarray*}
hence~$\F(G)=\phi(m)$ and so~$F:=G$ is the desired minimizer.

In the case $0=s<t\leq 1$, our problem reduces to the (fractional) isoperimetric problem, hence it is well known that there exists a minimizer $F$ for \eqref{problem} and it is a ball of volume $m$, for any $m>0$. 

When $0<s<t=1$ the previous arguments can be easily adapted, 
including the analog of Lemma \ref{ID} which becomes an ordinary differential inequality, 
and the only difference is that one needs to use the compact 
embedding of $BV$ into $H^\frac{s}{2}$ for $0<s<1$.
\end{proof}
 
\section{Regularity of minimizers}\label{section4} 

The aim of this section is 
to prove the regularity and rigidity theory
necessary to prove the second statement in Theorem~\ref{thexist}
and Theorem~\ref{thmain}.
We begin with a
simple observation.
\begin{lemma}\label{lem0}
Let $\phi$ be the function describing problem \eqref{problem}. Then $F$ is a minimizer of  $\phi(m)$ if and only
if $F/m^{1/N}$ is a minimizer of problem
\[
\min\{(1-t)P_t(U)-m^{\frac{t-s}{N}}sP_s(U):|U|=1\}. 
\]
\end{lemma}
\begin{proof}
Let $F\subseteq\R^N$ such that $|F|=m$ and let $U=F/m^{1/N}$. Then
\begin{eqnarray*}
  (1-t)P_t(F)-sP_s(F)&=&(1-t)P_t(m^{1/N}U)-sP_s(m^{1/N}U)\\
  &=&m^{\frac{N-t}{N}}\left[(1-t)P_t(U)-m^{\frac{t-s}{N}}sP_s(U)\right],
 \end{eqnarray*}
which gives the desired result.
\end{proof}
The previous lemma allow us to consider, in what follows, the functional
\begin{equation*}
 \Feps=(1-t)P_t-\eps s P_s,
\end{equation*}
where we set $\eps=m^{(t-s)/N}$. Indeed, the behavior of a minimizer of $\phi(m)$ is the same, up to a
rescaling, to that of
\begin{equation}\label{Feps}
\min\left\{\Feps(E):|E|=\omega_N\right\}.
\end{equation}
Indeed
\begin{equation}\label{scale}\begin{split}
&{\mbox{$F$ is a minimizer for problem~\eqref{Feps} if and only if}}\\
&{\mbox{$\left(\frac{m}{\omega_N}\right)^{\frac1N}F$ is a minimizer for problem~\eqref{problem}
with~$\eps=\left(\frac{m}{\omega_N}\right)^{\frac{t-s}N}$.}}
\end{split}\end{equation}
The next lemma allows us to say that if $F$ is a set of $\mathbb{R}^N$ such that $||F|-\omega_N|$ is small enough
than the volume constraint can be dropped. Let us consider the following problem:
\begin{equation}\label{Geps}
\min\left\{\mathcal{G}_{\eps,\Lambda}(E):||E|-\omega_N|<\Lambda^{-1}\right\},
\end{equation}
for some $\Lambda>0$, where
\begin{equation*}
\mathcal{G}_{\eps,\Lambda}(E)=(1-t)P_t(E)-\eps s P_s(E)+\Lambda ||E|-\omega_N|.
\end{equation*}

Letting 
\begin{equation}\label{e0}
\eps_0:=\left(\frac{\bar m_0}{\omega_N}\right)^\frac{t-s}{N},
\end{equation}
with $\bar m_0$ as in Theorem~\ref{PRE-thexist}, we have the following result:

\begin{lemma}\label{equiv}
There exists $\Lambda_0=\Lambda_0(N,\delta_0)>0$
such that $F_\varepsilon$ is a volume constrained minimizer 
of problem \eqref{Feps}, with $\eps< \eps_0$, 
if and only if $F_\varepsilon$ is a minimizer of problem \eqref{Geps}, 
for any $\Lambda\ge\Lambda_0(1+\eps_0)$.
\end{lemma}

\begin{proof}
First, let~$F_\varepsilon$ be
a minimizer of problem \eqref{Geps} with~$|F_\varepsilon|=\omega_N$.
Then, for any set~$G$ with~$|G|=\omega_N$, we have that
$$ \Feps(G)=\mathcal{G}_{\varepsilon,\Lambda}(G)
\ge \mathcal{G}_{\varepsilon,\Lambda}(F_\varepsilon)=
\Feps(F_\varepsilon),$$
which shows that~$F_\varepsilon$ is
a minimizer of problem \eqref{Feps}.
 
Viceversa, 
we prove that a volume constrained minimizer 
$F_\varepsilon$ of problem \eqref{Feps}, with $\eps<\eps_0$, 
is also a minimizer of \eqref{Geps} for any $\Lambda$ sufficiently large. 
For this, we argue by contradiction and 
we assume that there exist a sequence $\Lambda_n\to +\infty$,
and sets $E_n\subset\R^N$ such that,
letting $\G_n:=\G_{\eps,\Lambda_n}$, we have 
\begin{equation}\label{q1}
\G_n(E_n)<\G_n(F_\eps)=\Feps(F_\eps).
\end{equation}
Notice that for all $n\in\mathbb N$ there holds
\begin{equation}\label{q2}
\sigma_n:= \big||E_n|-\omega_N\big| > 0.
\end{equation}
Indeed, if by contradiction we suppose that~$\sigma_n=0$ 
for some $n\in\mathbb N$,
we would have that $|E_n|=\omega_N$, thus
$$\G_n(E_n)=\Feps(E_n)
\ge \Feps(F_\eps),
$$
due to the minimality of~$F_\eps$. This would be
in contradiction with~\eqref{q1}, and so~\eqref{q2}
is proved.
We also claim that there exists a constant $c_0>0$
independent of $n$, such that  
\begin{equation}\label{q4}
(1-t)P_t(E_n)\le c_0 \qquad \text{and}\qquad 
sP_s(E_n)\le c_0 \qquad \text{for all $n\in\mathbb N$.}
\end{equation}
To show this, proceeding as in Proposition~\ref{CC} and thanks to~\eqref{q1}, we see that, for $\Lambda_n\ge
c_0\,\eps_0^\frac{t}{t-s}$, we have
\begin{eqnarray*}
(1-t)P_t(E_n)&\le& 2\left[\Feps(E_n)+c_0 \eps_0^\frac{t}{t-s} |E_n|\right] 
\\
&\le& 2\left[\Feps(E_n)+c_0 \, \eps_0^\frac{t}{t-s} \big||E_n|-\omega_N\big|
+ \omega_N c_0 \eps_0^\frac{t}{t-s}\right]
\\
&\le& 2\left[\G_n(E_n)+\omega_Nc_0\, \eps_0^\frac{t}{t-s}\right]\\
&\le& 2\left[\Feps(F_\eps)+\omega_Nc_0\,\eps_0^\frac{t}{t-s}\right] \\
&\le& 2\left[\Feps(B)+\omega_Nc_0\,\eps_0^\frac{t}{t-s}\right] \\
&\le& 2\left[(1-t)P_t(B)+\omega_Nc_0\,\eps_0^\frac{t}{t-s}\right],
\end{eqnarray*}
recalling that $B$ denotes the ball centered in $0$ and radius $1$, with $|B(0,1)|=\omega_N$.
This gives the bound for $(1-t)P_t(E_n)$, and then the bound on~$sP_s(E_n)$
follows from~\eqref{basicestimates}. This proves~\eqref{q4}.

{}From \eqref{q1} and \eqref{q4} it follows
that $\Lambda_n\sigma_n$ is also uniformly bounded, that is,
\[
\sigma_n \le \frac{c_0}{\Lambda_n}\to 0\qquad 
\text{as }n\to +\infty.
\]
Moreover, for $\sigma_n\le 1/2$ we have, supposing $\sigma_n=|E_n|-\omega_N>0$ (the other case can be treated in a similar way),
\begin{equation}\label{qt1} 
 \left(\frac{ |E_n|}{\omega_N}\right)^{-\frac{N-s}{N}} \ge \left(1+\frac{\sigma_n}{\omega_N}\right)^{-\frac{N-s}{N}}
\ge 1 -\frac{N-s}{N}\frac{\sigma_n}{\omega_N}\,, 
\end{equation}
and similarly 
\begin{equation}\label{qt2}
 \left(\frac{ |E_n|}{\omega_N}\right)^{-\frac{N-t}{N}} \le \left(1-\frac{\sigma_n}{\omega_N}\right)^{-\frac{N-t}{N}}
\le 1+C\frac{N-t}{N}\frac{\sigma_n}{\omega_N},
\end{equation}
with $C=C(N,s,t)$.
We now define
$$
\tilde{E}_n= \left(\frac{ |E_n|}{\omega_N}\right)^{-\frac 1 N} E_n\,,
$$
and we use~\eqref{q4}, \eqref{qt1} and~\eqref{qt2} to obtain
\begin{eqnarray*}
&& sP_s(\tilde{E}_n)= \left(\frac{ |E_n|}{\omega_N}\right)^{-\frac{N-s}N}
sP_s(E_n)\geq \left(1-\frac{N-s}{N}\frac{\sigma_n}{\omega_N}\right)sP_s(E_n)
\ge sP_s(E_n)-c_0\sigma_n,
\\
{\mbox{and }}&& (1-t)P_t(\tilde{E}_n)=\left(\frac{ |E_n|}{\omega_N}\right)^{-\frac{N-t}N}
(1-t)P_t(E_n)
\le (1-t)P_t(E_n)+c_0\sigma_n\,,
\end{eqnarray*}
where the constant $c_0$ may differ from line to line.

Therefore, since~$|\tilde{E}_n|=\omega_N$, the minimality of~$F_\eps$ gives
\begin{eqnarray*} 
\Feps(F_\eps)&\le& \Feps(\tilde{E}_n)
\,=\,(1-t)P_t(\tilde{E}_n)-\eps s P_s(\tilde{E}_n)
\\
&\le& (1-t)P_t(E_n)-\eps s P_s(E_n)+c_0(1+\eps_0)\sigma_n
\\
&=&\Feps(E_n)+c_0(1+\eps_0)\sigma_n.
\end{eqnarray*}
By plugging this into~\eqref{q1} we find that
\begin{eqnarray*} 
\G_\eps(E_n)&<&\Feps(F_\eps)\le
\Feps(E_n)+c_0(1+\eps_0)\sigma_n\\
&=&\G_\eps(E_n)-\Lambda_n \sigma_n+c_0(1+\eps_0)\sigma_n.
\end{eqnarray*}
We simplify the term $\G_\eps(E_n)$ and we divide by~$\sigma_n$, 
which is possible thanks to~\eqref{q2}, we conclude that
$$ 0< -\Lambda_n + c_0(1+\eps_0).$$
This gives a contradiction for $\Lambda_n$ large enough,
and proves that~$F_\eps$ is a minimizer for problem~\eqref{Geps}.
\end{proof}

\begin{lemma}\label{lemuno}
Let $F_\varepsilon$ be a minimizer of problem \eqref{Geps} 
with $\varepsilon<\varepsilon_0$ and $\Lambda\ge \Lambda_0$, $\varepsilon_0$ and $\Lambda_0$ as in Lemma~\ref{equiv},
and let $E_\varepsilon$ be a set of finite perimeter with $\big||E_\varepsilon|-\omega_N\big|<1/\Lambda$. Then,
\begin{eqnarray}\label{stimamin}
\nonumber(1-t)P_t(F_\varepsilon)&\le &(1-t)P_t(E_\varepsilon) + \varepsilon \,c_N\left(1-\frac st\right)^{-1}
|F_\varepsilon\Delta E_\varepsilon|^{1-\frac st}[(1-t)P_t(F_\varepsilon\Delta E_\varepsilon)]^{\frac st}\\
&&+\Lambda\big||E_\varepsilon|-\omega_N\big|\big|.
\end{eqnarray}
\end{lemma}

\begin{proof}
Notice that, denoting by $\int_U=\int_U f$ for a non-negative function $f$,  the following computation holds
\[
 \int_{F_\varepsilon}\int_{F_\varepsilon^c}=
\int_{F_\varepsilon\setminus E_\varepsilon}\int_{(F_\varepsilon\cup E_\varepsilon)^c}+
\int_{F_\varepsilon\setminus E_\varepsilon}\int_{E_\varepsilon\setminus
F_\varepsilon}+\int_{F_\varepsilon\cap E_\varepsilon}\int_{(F_\varepsilon\cup
E_\varepsilon)^c}+\int_{F_\varepsilon\cap E_\varepsilon}\int_{E_\varepsilon\setminus F_\varepsilon}.
\]
By interchanging the roles of $F_\varepsilon$ and $E_\varepsilon$, and setting $f(x,y)=|x-y|^{-(s+N)}$ we get
\begin{equation}\label{P21}
\begin{aligned}
P_s(F_\varepsilon)-P_s(E_\varepsilon)&=
\int_{F_\varepsilon\setminus E_\varepsilon}\int_{(F_\varepsilon\cup E_\varepsilon)^c}-\int_{E_\varepsilon\setminus
F_\varepsilon}\int_{(F_\varepsilon\cup E_\varepsilon)^c}
+\int_{F_\varepsilon\cap E_\varepsilon}
\int_{E_\varepsilon\setminus F_\varepsilon}-
\int_{E_\varepsilon\cap F_\varepsilon}
\int_{F_\varepsilon\setminus E_\varepsilon}
\\
 &\le \int_{F_\varepsilon\setminus
E_\varepsilon}\int_{(F_\varepsilon\cup E_\varepsilon)^c}+\int_{E_\varepsilon\setminus
F_\varepsilon}\int_{F_\varepsilon\cap E_\varepsilon}\le 
P_s(F_\varepsilon\Delta E_\varepsilon).
\end{aligned}
\end{equation}
Therefore, by the minimality of $F_\varepsilon$ we get  
\[
\begin{aligned}
(1-t)P_t(F_\varepsilon)&\le (1-t)P_t(E_\varepsilon) + \varepsilon \left[sP_s(F_\varepsilon) -
sP_s(E_\varepsilon)\right] 
+ \Lambda\Big(\big||E_\varepsilon|-\omega_N\big|-\big||F_\varepsilon|-\omega_N\big|\Big)
\\
&\le (1-t)P_t(E_\varepsilon)+ \varepsilon s P_s(F_\varepsilon\Delta E_\varepsilon)+
\Lambda\big||E_\varepsilon|-\omega_N\big| .
\end{aligned}
\]
Hence the desired result follows from~\eqref{basicestimates}.
\end{proof} 

We point out that from Lemma \ref{lemuno}
it follows that $F_\varepsilon$ is a multiplicative $\omega$-minimizer for the $t$-perimeter.
In the sequel, as customary, the fractional perimeter of a set~$E$ in a ball~$B(x,R)$ will be denoted
by~$P_t(E,B(x,R))$.

\begin{corollary}\label{correg}
Let~$\eps_0$ and~$\Lambda_0$ be as in Lemma \ref{equiv}.
Let $F_\varepsilon$ be a minimizer of \eqref{Feps} with $\varepsilon<\varepsilon_0$, let $x\in\partial^m
F_\varepsilon$, 
and let $E_\varepsilon$ be a set of finite $t$-perimeter with 
\begin{equation}\label{Con}
F_\varepsilon\Delta E_\varepsilon\subset B(x,R).\end{equation}
There holds
\begin{equation}\label{stimaomega}
P_t(F_\varepsilon,B(x,R))\le \frac{1+c_0R^{t-s}}{1-c_0R^{t-s}}\, P_t(E_\varepsilon,B(x,R))
\end{equation}
for some $c_0>0$ and for any $R<R_0=R_0(N,\delta_0)$. 
\end{corollary}

\begin{proof} 
We observe that, by direct calculations, from~\eqref{Con}, follows
\begin{equation}\label{A1}
\begin{split}
P_t(F_\varepsilon)-P_t(E_\varepsilon)
=P_t(F_\varepsilon,B(x,R))-P_t(E_\varepsilon,B(x,R))\\
{\mbox{and }} \qquad P_t(F_\varepsilon\Delta E_\varepsilon)\le
P_t(F_\varepsilon,B(x,R))+P_t(E_\varepsilon,B(x,R))
.\end{split}
\end{equation}
Furthermore, thanks to Lemma \ref{equiv} we know that $F_\eps$ is also a minimizer of \eqref{Geps}, with
$\Lambda=\Lambda_0$. 
From \eqref{stimamin} and the fractional isoperimetric inequality \eqref{iso_ineq2}, we then get
\begin{eqnarray}\label{stimamon}
\nonumber (1-t)P_t(F_\varepsilon,B(x,R))&\le& (1-t)P_t(E_\varepsilon,B(x,R)) \\
\nonumber && +  \varepsilon_0\, c_N\left(1-\frac st\right)^{-1}[c_N\,t]^{1-\frac st} |F_\varepsilon\Delta
E_\varepsilon|^\frac{t-s}{N} (1-t)P_t(F_\varepsilon\Delta E_\varepsilon)\\
&& +\Lambda_0\big||E_\varepsilon|-\omega_N\big|.
\end{eqnarray}
Moreover, again from the fractional isoperimetric inequality and using \eqref{Con},
\begin{eqnarray*}
\Lambda_0 \big||E_\varepsilon|-\omega_N\big|
&=&
\Lambda_0\Big|\big||E_\varepsilon|-\omega_N\big|-\big||F_\varepsilon|-\omega_N\big|\Big| 
\\
&\leq& \Lambda_0 |F_\varepsilon\Delta E_\varepsilon|^{\frac{N-t}{N}} |F_\varepsilon\Delta
E_\varepsilon|^{\frac{t}{N}}  \\
&\leq& c_N\, \Lambda_0\,t\,(1-t) P_t(F_\varepsilon\Delta E_\varepsilon)\,|F_\varepsilon\Delta
E_\varepsilon|^{\frac{t}{N}}\\
&\leq& c_N\, \Lambda_0\,t\,(1-t) P_t(F_\varepsilon\Delta E_\varepsilon)  R^t.
\end{eqnarray*}
{F}rom this, \eqref{A1} and~\eqref{stimamon}
we arrive at
\begin{eqnarray*}
(1-t)P_t(F_\varepsilon,B(x,R)) &\le& (1-t)P_t(E_\varepsilon,B(x,R)) \\ && + \varepsilon_0\, c_N\left(1-\frac
st\right)^{-1}[c_N\,t]^{1-\frac st} R^{t-s} (1-t)P_t(F_\varepsilon\Delta E_\varepsilon)\\
&& + \Lambda_0 \,c_N\, tR^t\,(1-t) P_t(F_\varepsilon\Delta E_\varepsilon) \\ &\le& (1-t)P_t(E_\varepsilon,B(x,R)) \\
&&+ c_0 R^{t-s}(1-t)\left[P_t(F_\varepsilon,B(x,R))+P_t(E_\varepsilon,B(x,R))\right]
\end{eqnarray*}
which gives \eqref{stimaomega}, if $R<\min\{1,1/c_0^{\frac{1}{t-s}}\}=:R_0$, with 
$$
c_0:= \varepsilon_0\, c_N\,\left(1-\frac st\right)^{-1}[c_N\,t]^{1-\frac st} +\Lambda_0\, c_N\, t.
$$
\end{proof}

\begin{lemma}\label{lemdens}
There exists $\Theta=\Theta(N,\delta_0)>0$ and $R_0=R_0(N,\delta_0)>0$ such that,
for any $x\in\partial^m F_\eps$ and $R<R_0$, there holds 
\begin{equation}\label{stimasup}
(1-t)P_t(F_\eps,B(x,R))\le \Theta\,R^{N-t}.
\end{equation}
\end{lemma}

\begin{proof}
Let $E_\eps=F_\eps\setminus B(x,R)$, 
and observe that $P_t(E_\eps,B(x,R))\le P_t(B(x,R))$.
From \eqref{stimaomega}, possibly reducing $R_0$, 
we then get
\[
(1-t)P_t(F_\eps,B(x,R))\le (1+c_0\,R^{t-s})(1-t)P_t(B(x,R)) \le \Theta\,R^{N-t}.
\]
\end{proof}

From Lemma \ref{lemdens} it follows that $F_\eps$ is also an additive $\omega$-minimizer for the
$t$-perimeter.
 
\begin{corollary}\label{lemadd}
Let~$\eps_0$ be as in Lemma \ref{equiv}.
Let $F_\varepsilon$ be a minimizer of \eqref{Feps} with $\varepsilon<\varepsilon_0$, let $x\in\partial^m
F_\varepsilon$, 
and let $E_\varepsilon$ be a set of finite $t$-perimeter with 
\begin{equation}
F_\varepsilon\Delta E_\varepsilon\subset B(x,R).\end{equation}
There holds
\begin{equation}\label{stimaomegaadd}
(1-t)P_t(F_\varepsilon,B(x,R))\le (1-t)P_t(E_\varepsilon,B(x,R)) + c_0\,R^{N-s}
\end{equation}
for any $R<R_0$, with~$R_0,c_0$ depending only on $N,\,\delta_0$.
\end{corollary}

\begin{proof}
By \eqref{stimaomega} and \eqref{stimasup}, possibly increasing the constant $c_0$ we have
\[
(1-t)P_t(E_\varepsilon,B(x,R))\ge \left( 1-c_0R^{t-s}\right) 
(1-t)P_t(F_\varepsilon,B(x,R))
\ge (1-t)P_t(F_\varepsilon,B(x,R)) - c_0\Theta\, R^{N-s}
\]
for any $R<R_0$.
\end{proof}

The following result can be proved exactly as in \cite[Theorem 3.4]{labandadei5}.

\begin{proposition}\label{proflat}
Let $t_0\in (0,1)$, there exist $\tau,\delta,q \in (0,1)$, depending only on $N, t_0$, such that if $F$ is an
additive $\omega$-minimizer of $P_t$ for any $t \in [t_0, 1]$,
%
with 
$0\in\partial^m F$ and 
\[
\partial^m F\cap B(0,1)\subset \left\{ y\in\R^N:\ |(y-x)\cdot e|<\tau\right\},
\]
for some $e\in S^{N-1}$,
then there exists $e_0\in S^{N-1}$ such that 
\[
\partial^m F\cap B(0,\eta)\subset \left\{ y\in\R^N:\ |(y-x)\cdot e_0|<q\delta \tau\right\}.
\]
\end{proposition}

From Corollary \ref{lemadd} and Proposition \ref{proflat} we derive the $C^{1,\beta}$ regularity 
minimizer of \eqref{Feps} following standard arguments that can be found in \cite[Theorem 1]{CG11} 
(see also \cite[Corollary 3.5]{labandadei5}).

\begin{corollary}\label{corregbis}
There exists $\beta=\beta(N,\delta_0)<1$ such that any minimizer
$F_\eps$ of \eqref{Feps}, with $\eps<\eps_0$, as in Lemma \ref{equiv}, has boundary of class $C^{1,\beta}$
outside of a closed singular set of Hausdorff dimension at most $N-2$.
\end{corollary}

\begin{remark}\label{remreg}\rm
If $t=1$, by the general regularity theory for $\omega$-minimizers
of the classical perimeter developed in \cite{bombieri, tamanini} 
we have that $F_\eps$ has boundary of class $C^{1,\beta}$ outside of a closed singular set of Hausdorff dimension
at most $N-8$.
\end{remark}

We are in the position of completing the
proof of Theorem~\ref{thexist}.\medskip

\noindent{\it Proof of Theorem \ref{thexist}.}
The existence follows from Theorem~\ref{PRE-thexist}.
The regularity of $\partial F$ follows from Corollary \ref{correg} and
Corollary \ref{corregbis}. \qed

\section{Rigidity of minimizers for small volumes}\label{section5}

We now develop the rigidity theory needed to prove Theorem~\ref{thmain}.

\begin{theorem}\label{teoregolarita}
For any $\eta>0$ there exists $\bar\eps=\bar\eps(\eta,N,\delta_0)>0$
such that any minimizer $F_\eps$ of \eqref{Feps},
with $\eps<\bar\eps$, can be written as 
 \begin{equation}\label{eqgraph}
  \partial F_\eps=\{(1+u_\eps(x))x:x\in\partial B\},
 \end{equation}
where $B$ is the ball of radius 1
having the same barycenter of $F_\eps$,
and $u_\eps:\partial B\to \R$ satisfies
\[
\|u_\eps\|_{C^1(\partial B)}\le \eta.
\]
\end{theorem}

\begin{proof}
From Lemma \ref{deficit}, putting $m=\varepsilon^{\frac{N}{t-s}}\omega_N$ there, it follows that $|F_\eps\Delta B|\to 0$ as $\eps\to 0$.
From the density lower bound proved in Proposition \ref{densityestimate} it then follows that 
$\partial F_\eps\to \partial B$ in the Hausdorff topology. 
The result now follows via a standard argument based on the $\omega$-minimality of $F_\eps$
and on the regularity of the limit set $B$ (see \cite[Corollary 3.6]{labandadei5} and, for $t=1$, 
\cite[Theorem 1]{tamanini} and \cite[Theorem 26.6]{M12}).
\end{proof}

\begin{theorem}\label{teodelta}
There exist $\tau_0,c_1,c_2>0$ depending only on $N$, with $c_1<c_2$, 
with the following property.
Suppose that $E_\tau$ is such that, for $\tau\in[0,\tau_0]$,
$\partial E_\tau$ takes the form
 \[
  \partial E_\tau=\{(1+\tau u(x))x:x\in\partial B\},
 \]
where $u:\partial B\to \R$ satisfies
\[
 \|u\|_{C^1(\partial B)}\le 1/2.
\]
Suppose moreover that the barycenter of $E_\tau$ is the same of that of $B$, 
say $0$, and that $|E_\tau|=|B|$. Then, for all $\alpha\in (0,1)$ 
it holds true that
\begin{equation}\label{claimrego}
c_1\, \tau^2 \left([u]^2_{H^{\frac{1+\a}{2}}(\partial B)}
+\a P_\a(B)\|u\|^2_{L^2(\partial B)}\right)
\,\le\,
P_\a(E_\tau)-P_\a(B)\,\le\, c_2\, \tau^2 [u]^2_{H^{\frac{1+\a}{2}}(\partial B)}.
\end{equation}
\end{theorem}

\begin{proof}
The first inequality in \eqref{claimrego} has been proved
in \cite[Theorem 2.1]{labandadei5}. It remains to prove the 
second inequality.

As in \cite[Formula (2.20)]{labandadei5},
after some calculations we get that
\begin{equation}\label{one}
 P_\a(E_\tau)=\frac{ \tau^2}{2}g(\tau)+\frac{P_\a(B)}{P(B)}h(\tau),
\end{equation}
where we set
\[
 h(\tau):=\int_{\partial B}(1+\tau u(x))^{N-\a}\,d\H^{N-1}(x),
\]
\[
{\mbox{and }} \qquad g(\tau):=\int_{\partial
B}\int_{\partial
B}\left(\int_{u(y)}^{u(x)}\int_{u(y)}^{u(x)}f_{|x-y|}(1+\tau r,1+\tau\rho)
\,dr\,d\rho\right)\,d\H^{N-1}(x)\,d\H^{N-1}(y),
\]
being
\begin{equation}\label{df}
 f_\theta(a,b):=\frac{a^{N-1}b^{N-1}}{(|a-b|^2+ab\theta^2)^{\frac{N+\a}{2}}}.
\end{equation}
We observe that $r$ and~$\rho$ in the definition of~$g$
range in~$[-\|u\|_{L^\infty(\partial B)},\,\|u\|_{L^\infty(\partial B)}]
\subseteq [-1,1]$, since~$\|u\|_{L^\infty(\partial B)}\le1$.
Hence, comparing with the definition of~$g$,
we notice that~$a$ and~$b$ in~\eqref{df}
range in~$[1-\tau,
1+\tau]$, and therefore they are bounded
and bounded away from zero. As a consequence, we get
$$ f_\theta(a,b)\le \frac{C_1}{(C_2+C_3\theta^2)^{\frac{N+\a}{2}}}
\le \frac{C_1}{(C_3\theta^2)^{\frac{N+\a}{2}}}=\frac{C_4}{\theta^{N+\a}},$$
for suitable constants~$C_1,\dots,C_4>0$.
Therefore, up to renaming the constants, we have
$$ g(\tau)\le 
\int_{\partial
B}\int_{\partial
B}\left(\int_{u(y)}^{u(x)}\int_{u(y)}^{u(x)}
\frac{c_N}{|x-y|^{N+\a}}
\,dr\,d\rho\right)\,d\H^{N-1}(x)\,d\H^{N-1}(y)=
c_N\,[u]^2_{H^{\frac{1+\a}{2}}(\partial B)}\,.
$$
Thus, since $h(0)=P(B)$, by \eqref{one} we get
\begin{equation}\label{two}
 P_\a(E_\tau)-P_\a(B)\le c_N\, \tau^2[u]^2_{H^{\frac{1+\a}{2}}(\partial B)}
+\frac{P_\a(B)}{P(B)}(h(\tau)-h(0)).
\end{equation}
Now we want to estimate $h(\tau)-h(0)$. 
Since $|E_\tau|=|B|$, using polar coordinates,
we get 
\begin{equation}\label{VC}
 \int_{\partial B}(1+\tau u)^N\,d\H^{N-1}=N|E_\tau|=N|B|=P(B).
\end{equation}
Thus
\begin{equation}\label{twobis}
 h(\tau)-h(0)=
 \int_{\partial B}(1+\tau u)^{N-\a}\,d\H^{N-1}-P(B)=
\int_{\partial B}(1+\tau u)^N((1+\tau u)^{-\a}-1)\,d\H^{N-1}.
\end{equation}
By a Taylor expansion, we know that for any 
$x\ge0$ small enough, it holds 
\begin{eqnarray*}
 &&\left((1+x)^{-\a}-1)(1+x)^N\right)
\\ &&\qquad=\left(-\a x+\frac{\a(\a+1)}{2}x^2+\a 
\beta(x)\right)\left(1+Nx+\frac{N(N-1)}{2}
x^2+\gamma(x)\right) ,
\end{eqnarray*}
with $|\beta(x)|+|\gamma(x)|\le c_N x^3$, so that
 \[
  \left((1+x)^{-\a}-1)(1+x)^N\right)
\le-\a x+\left(\frac{\a(\a+1)}{2}-N\a\right)x^2+\a\, c_N\,x^3.
 \]
By applying such an inequality to \eqref{twobis}, 
and using the fact that $\|u\|_{L^\infty(\partial B)}<1$, we get
\begin{equation}\label{VC2}
 h(\tau)-h(0)\le-\a\int_{\partial B}\left[\tau u+\left(N-\frac{\a+1}{2}\right)
\tau^2u^2\right]\,d\H^{N-1}
+\a \,c_N\,\tau ^3\|u\|_{L^2(\partial B)}^2.
\end{equation}
Also, from~\eqref{VC}, we have
$$
0=\int_{\partial B}((1+\tau u)^N-1)\,d\H^{N-1}\le
\int_{\partial B}(N\tau u+N(N-1)\tau^2u^2+c_N\,\tau^3u^3)\,d\H^{N-1}.
$$
Hence, since $\|u\|_{L^\infty(\partial B)}< 1$, 
we obtain
\[
-\int_{\partial B}\tau u\,d\H^{N-1}\le \frac{N-1}{2}\, \tau^2\|u\|^2_{L^2(\partial B)}
+c_N\,\tau^3\|u\|^2_{L^2(\partial B)}\,,
\]
so that~\eqref{VC2} gives 
\[
 h(\tau)-h(0)\le -\frac{\tau^2}{2}\a\,(N-\a)\|u\|^2_{L^2(\partial B)}
+\a \, c_N\,\tau^3\|u\|^2_{L^2(\partial B)} \le0
\]
for $\tau\le\tau_0(N)$. 
By inserting this into \eqref{two} we obtain
the second inequality in~\eqref{claimrego}.
\end{proof}

\smallskip

We now complete the proof of Theorem \ref{thmain}.
\medskip

\noindent{\it Proof of Theorem \ref{thmain}.} 
We have to show that there exists $\eps_1=\eps_1(N,\delta_0)\in (0,\eps_0]$, $\eps_0$ as in \eqref{e0}, and so $\bar m_1=\bar m_1(N,\delta_0)\in (0,\bar m_0]$, such that the ball $B$ is the only minimizer of 
problem \eqref{Feps} for $\eps<\eps_1$.
Let $\eps<\eps_1$ and let $F_\eps$ be a minimum of problem \eqref{Feps}, which exists by Theorem \ref{thexist}.
By the minimality of $F_\eps$ we have
\begin{equation}\label{claim}
(1-t)P_t(F_\eps)-(1-t)P_t(B)\le \eps \left(sP_s(F_\eps)-sP_s(B)\right)
\end{equation}
where $B$ has the same barycenter of $F_\eps$.
Possibly reducing $\eps$ we can assume that $\partial F_\eps$ can be written as in \eqref{eqgraph},
with $\|u_\eps\|_{C^1(\partial B)}\le \tau_0/2$, where $\tau_0$ is as in Theorem \ref{teodelta}. Then, from
\eqref{claim} and \eqref{claimrego} it follows 
\begin{eqnarray}\label{lippa}
\nonumber c_1 (1-t)[u_\eps]^2_{H^{\frac{1+t}{2}}(\partial B)}&\le& c_1
(1-t)\left([u_\eps]^2_{H^{\frac{1+t}{2}}(\partial B)}+tP_t(B)\|u_\eps\|^2_{L^2(\partial B)}\right)
\\
\nonumber &\le& \big((1-t)P_t(F_\eps)-(1-t)P_t(B)\big)
\\
\nonumber &\le& \eps \left(sP_s(F_\eps)-sP_s(B)\right)
\\
&\le& \eps\,s\,c_2 
\,[u_\eps]^2_{H^{\frac{1+s}{2}}(\partial B)}.
\end{eqnarray}
From \eqref{H} it then follows 
\[
c_1 (1-t)[u_\eps]^2_{H^{\frac{1+t}{2}}(\partial B)}\leq 
c_N  \frac{\varepsilon s}{(1-s)}\,(1-t)[u_\eps]^2_{H^{\frac{1+t}{2}}(\partial B)}
\]
which implies $u_\eps=0$, that is $F_\eps=B$, whenever $\eps$ is sufficiently small.~\qed

We conclude the section with the following counterpart to Theorem \ref{thmain}.

\begin{theorem}\label{prounstable}
For all $0<s<t\le 1$,
there exists a volume $\bar m_2=\bar m_2(N,s,t)\ge \bar m_1$ such that, 
for $m>m_2$, the ball is not a local minimizer of problem \eqref{problem}.
\end{theorem}

\begin{proof}
We have to show that there exists $\eps_2\ge\eps_1$
such that the ball $B$ is not a local minimizer of problem \ref{Feps} 
for $\eps>\eps_2$. 
We look for a competitor $F_\eps\neq B$ which can be written as in \eqref{eqgraph},
with $u\not\equiv 0$ and 
and $\|u\|_{C^1(\partial B)}\le \tau_0/2$, where $\tau_0$ 
is as in Theorem \ref{teodelta}. As above, 
from \eqref{claimrego} it follows 
\begin{eqnarray}\label{loppa}
\nonumber \big((1-t)P_t(F_\eps)-(1-t)P_t(B)\big) &\le&
c_2 (1-t)[u]^2_{H^{\frac{1+t}{2}}(\partial B)}
\\
&<& \eps c_1 s [u]^2_{H^{\frac{1+s}{2}}(\partial B)}
\\
\nonumber &\le& \eps \left(sP_s(F_\eps)-sP_s(B)\right),
\end{eqnarray}
as soon as 
\[
\eps > \eps_2 := \frac{c_2 (1-t)[u]^2_{H^{\frac{1+t}{2}}(\partial B)}}
{c_1 s[u]^2_{H^{\frac{1+s}{2}}(\partial B)}}\,.
\]
This shows that $F_\eps$ has lower energy than $B$, so that the ball 
cannot be a local minimizer of problem \eqref{problem}.
\end{proof}

Notice that $\lim_{s\to 0}\bar m_2(N,s,t)=+\infty$ for all $t\in (0,1]$,
which is consistent with the fact that the ball is the unique 
minimizer of the $t$-perimeter, with volume constraint.

\section{A fractional isoperimetric problem}\label{section6}
\noindent
We recall from the Introduction the definition of the functional $\widetilde\F$ given by
\[
\widetilde\F(E)=
\begin{cases}
\frac{\left((1-t)P_t(E)\right)^{N-s}}{\left(sP_s(E)\right)^{N-t}} & \text{if $0<s<t<1$}\\
 \quad& \quad\\
\frac{(N\omega_{N}P(E))^{N-s}}{\left(sP_s(E)\right)^{N-1}} & \text{if $0<s<t=1$}\\
\quad& \quad\\
\frac{(1-t)P_t(E)^{N}}{(N\omega_{N}|E|)^{N-t}} & \text{if $0=s<t<1$}\\
\quad &\quad\\
N\omega_{N}\frac{P(E)^{N}}{|E|^{N-1}}&\text{if $s=0$ and $t=1$}.
\end{cases}
\]
In this section we consider the generalized isoperimetric problem 
\begin{equation}\label{isoper}
\min_{E\subset\R^N}\,\widetilde\F(E),
\qquad 0\le s<t\le 1\,.
\end{equation}
\begin{remark}\rm
Notice that the quantity in \eqref{isoper} is scale invariant,
hence without loss of generality we can look for minimizers 
$E$ satisfying a volume constraint $|E|=\omega_N$.
\end{remark}



The main aim of this section is  the following existence theorem.

\begin{theorem}\label{mainisop}
There exists a minimizer of problem \eqref{isoper}.
\end{theorem}

\noindent To prove it, we need some preliminary results,
namely a suitable version of the isoperimetric inequality
(Lemma~\ref{propiso1}) and an existence result with uniform
estimates for
a constrained minimization problem (Lemma~\ref{lemBR}).

\begin{remark}\label{notazione}\rm
In what follows, with a slight abuse of notation,
we extend the functionals $(1-t)P_t(\cdot)$ and $sP_s(\cdot)$
to $t=1$ and $s=0$ respectively, meaning
that for $t=1$ it equals $N\omega_{N}P(\cdot)$,
while for $s=0$ it equals $N\omega_{N}|\cdot|$. 
\end{remark}

\begin{lemma}\label{propiso1}
Let $s<t\in [0,1]$ satisfy \eqref{deltazero}. For any $E\subset\R^N$ there holds
\begin{equation}\label{N.A1}
\frac{\left((1-t)P_t(E)\right)^\frac{N-s}{N-t}}{\left(sP_s(E)\right)}\geq c
\end{equation}
for some $c=c(N,\delta_0)>0$.
\end{lemma}

\begin{proof}
Let $s<t\in [0,1]$ and let $\delta_0=t-s$.
Notice that
\begin{equation}\label{N.A2}
\frac{t}{t-s}=1+\frac{s}{t-s}\le1+\frac{1}{\delta_0}.
\end{equation} 
Then from \eqref{fracisop}, and since $\delta_0<t$, it follows
\[
|E|^{1-\frac{s}{t}}\le C(N,\delta_0) \left((1-t)P_t(E)\right)^\frac{N(t-s)}{(N-t)t}.
\]
Plugging this estimate into \eqref{basicestimates} (or \eqref{basicestimate} if $t=1$) we get
\[
sP_s(E)\ \le\ C(N,\delta_0)\,
\frac{t}{t-s}\, 
\left((1-t)P_t(E)\right)^\frac{N-s}{N-t}\,,
\]
which, together with~\eqref{N.A2} gives~\eqref{N.A1}.
\end{proof}

\noindent We notice that,
if $s=0$, the claim is an immediate  consequence of the the fractional isoperimetric inequality \eqref{fracisop}.

\begin{lemma}\label{lemBR}
Let $s<t\in [0,1]$ satisfy \eqref{deltazero}. 
For $R>1$ let $Q_R=[-R,R]^N$. Then, there exists a minimizer
$E_R$ of the problem
\begin{equation}\label{isoperBR}
\min_{E\subset Q_R\,|E|=m}\,\frac{((1-t)P_t(E))^{N-s}}{(sP_s(E))^{N-t}}\,.
\end{equation}
Moreover
\begin{equation}\label{conBR}
(1-t)P_t(E_R)\le C
\end{equation}
where $C$ is independent of $R$.
\end{lemma}

\begin{proof}
We recall that, thanks to the notation introduced in Remark
\ref{notazione} we can deal at once with the cases $t<1$ and $t=1$. 
By Lemma \ref{propiso1} we know that
\[
 C(R)=\inf_{E\subset
Q_R\,|E|=m}\,\frac{\left((1-t)P_t(E)\right)^\frac{N-s}{N-t}}{\left(sP_s(E)\right)}
\]
is a strictly positive quantity. Clearly the map $R\mapsto C(R)$ is 
non-increasing. Let $C=C(1)+1$ and
let $E_n$ be a minimizing sequence for \eqref{isoperBR}, 
so that for $n$ big enough it holds $(1-t)P_t(E_n)\le C (sP_s(E_n))^{(N-t)/(N-s)}$.
Possibly increasing the constant $C$,
from \eqref{basicestimates} (or \eqref{basicestimate} if $t=1$) it follows 
\[
(1-t)P_t(E_n)\le C ((1-t)P_t(E_n))^\frac{s(N-t)}{t(N-s)}
\]
which gives
\begin{equation}\label{ciccia}
(1-t)P_t(E_n)\le C \qquad \text{for all }n.
\end{equation}
The existence of a minimizer now follows by the direct method 
the calculus of variations, since the compact embedding of $L^1(Q_R)$ and $H^s(Q_R)$ into $H^t(Q_R)$ and the
estimate \eqref{conBR}
directly follows from \eqref{ciccia}.  
\end{proof}

We now prove Theorem \ref{mainisop}.

\begin{proof}[Proof of Theorem \ref{mainisop}]
If $s=0$ then the claim of the theorem is equivalent to that of the isoperimetric inequality (the fractional
isoperimetric inequality if $t<1$). Thus we consider just the case $s>0$. Again, we shall always write $(t-1)P_t$
meaning that such a functional is equivalent to the classical perimeter if $t=1$ (see Remark \ref{notazione}).

\noindent
Let $E_n$ be a minimizer of \eqref{isoperBR} with $R=n\in\mathbb N$ and $m=1/2$.
We divide $Q_n$ into $(2n)^N$ unit cubes with vertices in $\mathbb Z^N$,  
and we let $\{Q_{i,n}\}_{i=1}^{I_n}$ be the unit cubes with non-negligible 
intersection with $E_n$, that is, $x_{i,n}=|E_n\cap Q_{i,n}|\in (0,1/2]$
for all $i\in\{1,\ldots,I_n\}$, for some~$I_n\in\{1,\ldots,(2n)^N\}$. 

We remark that, from~\eqref{67-0} (and omitting the integrands for
simplicity), we have that
$$ \sum_{i=1}^\infty
P_t(E_n,Q_{i,n})=\sum_{i=1}^\infty
\int_{E_n\cap Q_{i,n}}\int_{\R^N\setminus E_n}
+\int_{Q_{i,n}\setminus E_n}\int_{E_n\setminus Q_{i,n}}
\le
\int_{E_n\cap Q_n}\int_{\R^N\setminus E_n}
+\int_{Q_{n}\setminus E_n}\int_{E_n},$$
which implies that
\begin{equation}\label{678.00}
\sum_{i=1}^\infty
P_t(E_n,Q_{i,n})\le 2\int_{E_n}\int_{\R^N\setminus E_n}
=2P_t(E_n).
\end{equation}
Now, up to reordering the cubes $Q_{i,n}$
we can assume that the sequence $\{x_{i,n}\}_{i=1}^{I_n}$ is non-increasing
in~$i$, and we set $x_{i,n}:=0$ for $i>I_n$. We have that
\begin{equation}\label{condx}
\sum_{i=1}^{\infty}x_{i,n}=\frac 12
\end{equation}
and, recalling \eqref{fracisoploc},
\eqref{conBR} and~\eqref{678.00}, and the fact that $x_{i,n}\le |E_n|=1/2=|Q_{i,n}|/2$, we get
\begin{equation}\label{condxx}
\sum_{i=1}^{\infty}x_{i,n}^\frac{N-t}{N}\le C
\sum_{i=1}^{\infty}(1-t)P_t(E_n,Q_{i,n})
\le 2C\,(1-t)P_t(E_n)\le C,
\end{equation}
up to renaming~$C$.
As in \cite[Lemma 4.2]{GN}, from \eqref{condx} 
and \eqref{condxx} it follows that
\begin{equation}\label{estx}
\sum_{i=k}^{\infty}x_{i,n}\ \le\ C\,k^{-\frac1N}
\end{equation}
for all $k\in\mathbb N$,
where $C$ depends only on $(N,s,t)$. 

Up to extracting a subsequence
(using either a diagonal process or Tychonoff Theorem), we can suppose that  
$x_{i,n}\to \alpha_i\in [0,1/2]$ as $n\to +\infty$ for every $i\in\mathbb N$, so that by \eqref{condx} and
\eqref{estx} we have
\begin{equation}\label{eqstella}
\sum_i \alpha_i=\frac 12\, .
\end{equation}
Fix now $z_{i,n} \in Q_{i,n}$. Up to extracting a further subsequence, we can suppose that $d(z_{i,n},z_{j,n}) \to
c_{ij} \in [0,+\infty]$, and (recalling~\eqref{conBR})
that there exists $G_i\subseteq\R^N$ such that
\begin{equation}\label{limit}
\left(E_{n}-z_{i,n}\right) \to G_i \quad 
\textrm{ in the } L^1_{\rm loc}\textrm{-convergence}
\end{equation}
for every $i\in\mathbb N$. 
We say that $i \sim j$ if $c_{ij} < +\infty$ and we denote by $[i]$ the equivalence class of $i$. Notice that $G_i$
equals $G_j$ up to a translation, if $i\sim j$. Let
$\mathcal A:=\{[i]: i\in\N\}$.
We claim that
\begin{equation}\label{limitbis}
\sum_{[i]\in\mathcal A} P_t(G_i)\le \liminf_{n\to +\infty} P_t(E_n)
\qquad \text{and} \qquad
\sum_{[i]\in\mathcal A} P_s(G_i)\le \liminf_{n\to +\infty} P_s(E_n)\,.
\end{equation}
To prove it, we first fix~$M\in\N$ and~$R>0$. 
We take different equivalent classes~$i_1,\dots,i_M$
and we notice that if~$i_k\ne i_j$ then the set~$z_{i_k,n}+Q_R$
is drifting far apart from~$z_{i_j,n}+Q_R$, and so
$$\lim_{n\to +\infty} 
\int_{z_{i_k,n}+Q_R}\int_{z_{i_j,n}+Q_R}\frac{dx\,dy}{|x-y|^{N+t}}=0.$$
Accordingly, by~\eqref{U35},
\eqref{limit}
and the lower semicontinuity of the perimeter,
\begin{eqnarray*}\sum_{i=1}^M P_t(G_i, Q_R)
&\le& \liminf_{n\to +\infty} \sum_{k=1}^M P_t(E_n, (z_{i_k,n}+Q_R))
\\&\le& \liminf_{n\to +\infty} P_t\left(E_n, \bigcup_{k=1}^N (z_{i_k,n}+Q_R)\right)
+ 2\sum_{{1\le k,j\le M}\atop{i_k\neq i_j}} 
\int_{z_{i_k,n}+Q_R}\int_{z_{i_j,n}+Q_R}\frac{dx\,dy}{|x-y|^{N+t}}
\\ &\le& \liminf_{n\to +\infty} P_t(E_n).\end{eqnarray*}
By sending first~$R\to +\infty$ and then~$M\to +\infty$, this yields~\eqref{limitbis}.

Now we claim that 
\begin{equation}\label{volume}
\sum_{[i]\in\mathcal A} |G_i|=\frac 12 .
\end{equation}
Indeed, for every $i\in\mathbb N$ and $R>0$ we have
\[
|G_i|\ge |G_i\cap Q_R| = \lim_{n\to +\infty} |(E_n-z_{i,n})\cap Q_R|.
\]
If $j$ is such that $j \sim i$ and $c_{ij} \le \frac{R}{2}$, 
possibly enlarging $R$ we have 
$Q_{j,n}- z_{i,n} \subset Q_R$ for all $n\in\mathbb N$, so that
\begin{eqnarray*}&& |(E_n-z_{i,n})\cap Q_R| =
\sum_{j=1}^{I_n}|(E_n-z_{i,n})\cap Q_R\cap (Q_{j,n}-z_{i,n})|
\\ &&\qquad\geq\sum_{j:\,c_{ij} \leq \frac{R}{2}} 
|(E_n-z_{i,n})\cap Q_R\cap (Q_{j,n}-z_{i,n})|
=\sum_{j:\,c_{ij} \leq \frac{R}{2}} 
|(E_n-z_{i,n})\cap (Q_{j,n}-z_{i,n})|\\ &&\qquad
=\sum_{j:\,c_{ij} \leq \frac{R}{2}} 
|E_n\cap Q_{j,n}|,\end{eqnarray*}
and so
\[
|G_i|
\ge\lim_{n\to +\infty} \left|\left(E_{n} -z_{i,n}\right) \cap Q_R\right|\geq 
\lim_{n \to +\infty} 
\sum_{j:\,c_{ij} \leq \frac{R}{2}} |E_{n} \cap Q_{j,n}|=\sum_{j:\,c_{ij} \leq \frac{R}{2}} \alpha_j.
\]
Letting $R\to +\infty$ we then have
\[
|G_i| \geq \sum_{j:\,i\sim j} \alpha_j=\sum_{j\in[i]}\alpha_j\,, 
\]
hence, recalling \eqref{eqstella},
\[
\sum_{[i]\in\mathcal A} |G_i| \ge \frac 12 ,
\]
thus proving \eqref{volume} (since the other inequality is trivial).

We now claim that 
\begin{equation}\label{limittris}
\sum_{[i]\in\mathcal A} P_s(G_i)\ge \limsup_{n\to +\infty} P_s(E_n).
\end{equation}
Indeed, by~\eqref{volume} we have that
for any $\eps>0$ there exist $R,\,\ell$ such that there exist $\ell$ distinct equivalence classes
$[i_1],\dots,[i_\ell]\in\mathcal A$ such that 
\begin{equation}\label{7899}
\frac 12 -\eps\le 
\sum_{k=1}^\ell|G_{i_k}\cap B_R|= \lim_{n\to +\infty}
\sum_{k=1}^\ell \left|\left(E_{n} -z_{i_k,n}\right) \cap B_R\right|.
\end{equation}
For $\rho>0$ we let
\[
E^\rho_{n,1}=E_n\cap \bigcup_{k=1}^\ell \left(z_{i_k,n}+B_\rho\right)
\qquad E^\rho_{n,2}=E_n\setminus E_{n,1}\,.
\]
For $n$ sufficiently large we have that the balls~$z_{i_k,n}+B_R$
are disjoint (since the $z_{i_k,n}$ are drifting far away from
each other, being each~$i_k$ in a different equivalence class). Therefore~\eqref{7899}
gives that
\begin{equation}\label{as2}
|E^R_{n,1}|\ge \frac 12 -2\eps \qquad\text{and}\qquad
|E^R_{n,2}|\le 2\eps
\end{equation}
if~$n$ is large enough.
We claim that 
\begin{equation}\label{eqdiambis}
\int_{E^{\bar\rho}_{n,1}}\int_{E^{\bar\rho}_{n,2}}\frac{dx\,dy}{|x-y|^{N+s}}
\le \frac{c_N}{s(1-s)}|E^{\bar\rho}_{n,2}|^\frac{N-s}{N}
\qquad\text{for some }\bar\rho\in \left[R,R+(2\delta)^{-\frac 1N}\right],
\end{equation}
where the constants $c_N,\,\delta$ 
depend only on $N$.

Indeed, if this is not the case, we would have that
\begin{equation}\label{eqdiambis2}
\begin{split}
&|E^{\rho}_{n,2}|>0 \ {\mbox{ and}}\\
&\int_{E^{\rho}_{n,1}}\int_{E^{\rho}_{n,2}}\frac{dx\,dy}{|x-y|^{N+s}}>
\frac{c_N}{s(1-s)}|E^{\rho}_{n,2}|^\frac{N-s}{N}
\qquad\text{for every }\rho\in \left[R,R+(2\delta)^{-\frac 1N}\right].\end{split}\end{equation}
So we let 
$$ \mu(\rho):=|E^\rho_{n,1}|=|E_n|-|E^\rho_{n,2}|=\frac12-|E^\rho_{n,2}|$$
and we obtain
\begin{eqnarray}\nonumber
\frac{c_N}{s(1-s)}\left(\frac 12-\mu(\rho)\right)^\frac{N-s}{N}
&<& \int_{E^\rho_{n,1}}\int_{E^\rho_{n,2}}\frac{dx\,dy}{|x-y|^{N+s}}
\\\label{lobel}
&\le& \int_{E^\rho_{n,1}}\int_{\R^N\setminus\bigcup_{k=1}^\ell \left(z_{i_k,n}+B_\rho\right)}
\frac{dx\,dy}{|x-y|^{N+s}}
\\\nonumber
&\le& \frac{N\omega_N}{s}\int_{E^\rho_{n,1}}\frac{dx}{(\rho-|x-z_{i_{k(x)},n}|)^s}
\\\nonumber
&=&\frac{N\omega_N}{s}\int_0^\rho\frac{\mu'(z)}{(\rho-z)^s}\,dz\,
\end{eqnarray}
for all $\rho\in \left[R,R+(2\delta)^{-\frac 1N}\right]$,
where $k(x)\in\mathbb N$ is such that $x\in z_{i_{k(x)},n}+B_\rho$.

{F}rom \eqref{lobel} and Lemma~\ref{ID} (used here with~$m:=1/2$
and~$\bar\rho:=R$), we obtain that
$\mu(\rho)=1/2$ (and so~$|E^\rho_{n,2}|=0$) for $\rho=R+(2\delta)^{-\frac 1N}$,
which leads to a contradiction with~\eqref{eqdiambis2}. 
We thus proved \eqref{eqdiambis}. Notice that inequality
\eqref{eqdiambis} holds also with $t$ instead of $s$.
So, by~\eqref{eqdiambis}
and the fact that $|E^{\bar\rho}_{n,2}|\le 2\eps$
(recall~\eqref{as2}), we obtain that
\begin{equation}\label{eqdiambis.1}
\int_{E^{\bar\rho}_{n,1}}\int_{E^{\bar\rho}_{n,2}}\frac{dx\,dy}{|x-y|^{N+s}}
\le C\eps^{\frac{N-s}{N}}
\ {\mbox{ and }} \
\int_{E^{\bar\rho}_{n,1}}\int_{E^{\bar\rho}_{n,2}}\frac{dx\,dy}{|x-y|^{N+t}}
\le C\eps^{\frac{N-t}{N}},
\end{equation}
for some~$C>0$, possibly depending on~$n$, $s$ and~$t$.

{}From this, \eqref{eqsum} and \eqref{conBR} we obtain
\begin{equation}\label{stoma0}\begin{split}
P_t(E^{\bar\rho}_{n,1})+
P_t(E^{\bar\rho}_{n,2})
\, &=P_t(E_n)
+2\int_{E^{\bar\rho}_{n,1}}\int_{E^{\bar\rho}_{n,2}}\frac{dx\,dy}{|x-y|^{N+t}}\\
&\le P_t(E_n)+C\eps^\frac{N-t}{N}\le C\,.
\end{split}\end{equation}
Now, by \eqref{basicestimates}, we get
$$ P_s(E^{\bar\rho}_{n,2})\le 
C\,|E^{\bar\rho}_{n,2}|^{1-\frac st}P_t(E_{n,2}^{\bar\rho})^{\frac st}\,,$$
up to renaming~$C$.
Using this, \eqref{stoma0} and then~\eqref{as2} once more, 
and possibly renaming~$C$ again, we conclude that
\begin{eqnarray*}
P_s(E^{\bar\rho}_{n,2}) &\le& 
C\,|E^{\bar\rho}_{n,2}|^{1-\frac st}
\Big(P_t(E_{n,1}^{\bar\rho})+P_t(E_{n,2}^{\bar\rho})\Big)^{\frac st}
\\ &\le& C\, |E^{\bar\rho}_{n,2}|^{1-\frac st}\\
&\le& C \eps^{1-\frac st}\,.\end{eqnarray*}
Consequently, using~\eqref{eqsum} and \eqref{eqdiambis.1}, we conclude that
\begin{equation}\label{stoma}
\begin{aligned}
P_s(E^{\bar\rho}_{n,1})&=P_s(E_{n})-P_s(E^{\bar\rho}_{n,2})+2\int_{E^{\bar\rho}_{n,1}}\int_{E^{\bar\rho}_{n,2}}
\frac {dx\,dy}{|x-y|^{N+s}}\\
&\ge P_s(E_n)-C\eps^{1-\frac st}.
\end{aligned}
\end{equation}

%
Also, from~\eqref{conBR}, \eqref{limit}
and the compact embedding of $H^\frac{t}{2}$ into $H^\frac{s}{2}$
(see \cite[Section 7]{DPV12}), we see that
\begin{equation}\label{N.102}
\lim_{n\to +\infty} P_s\left((E_n-z_{i_k,n})\cap B_{\bar\rho}\right)
=P_s(G_{i_k}\cap B_{\bar\rho}).\end{equation}


Now we recall that if $K$ is a convex set, then $P(E\cap K)\le P_t(E)$ (see for instance \cite[Lemma
$B.1$]{labandadei5}). Together with \eqref{R2} and~\eqref{stoma}, this implies
\begin{eqnarray*}
\sum_{[i]\in\mathcal A}P_s(G_i)&\ge& \sum_{k=1}^\ell P_s(G_{i_k}\cap 
B_{\bar\rho})
\\
&=& \lim_{n\to +\infty}\sum_{k=1}^\ell P_s\left((E_n-z_{i_k,n})\cap B_{\bar\rho}\right)
\\
&\ge& \lim_{n\to +\infty} P_s(E_{n,1}^{\bar\rho})
\\
&\ge&
\limsup_{n\to +\infty} P_s(E_n) - C(N,s,t)\eps^\frac{t-s}{t},
\end{eqnarray*}
which gives \eqref{limittris} by letting $\eps\to 0^+$.

{}From \eqref{limitbis} and \eqref{limittris} we obtain that
\begin{equation}\label{pag3}
\frac{\sum_{[i]\in\mathcal A} (1-t)P_t(G_i)}{\left(s\sum_{[i]\in\mathcal A} P_s(G_i)\right)^\frac{N-t}{N-s}}
\ \le\ \liminf_{n\to +\infty}\frac{(1-t)P_t(E_n)}{(sP_s(E_n))^\frac{N-t}{N-s}}\,.
\end{equation}

Let us now prove that the there exists $j$ such that
\begin{equation}\label{quasifine}
 \frac{(1-t)P_t(G_j)}{(sP_s(G_j))^\frac{N-t}{N-s}}
\le\frac{\sum_{[i]\in\mathcal A}
(1-t)P_t(G_i)}{\left(s\sum_{[i]\in\mathcal A} P_s(G_i)\right)^\frac{N-t}{N-s}}=:S.
\end{equation}
Indeed, if it is not the case, we get
\[
\begin{aligned}
 S&=\frac{\sum_{[i]\in\mathcal A} (1-t)P_t(G_i)}{\left(s\sum_{[i]\in\mathcal A}
P_s(G_i)\right)^\frac{N-t}{N-s}} 
= \frac{\sum_{[i]\in\mathcal A}
\left(\frac{(1-t)P_t(G_i)}{(sP_s(G_i))^\frac{N-t}{N-s}}\right)(sP_s(G_i))^\frac{N-t}{N-s}}{\left(s\sum_{[i]\in\mathcal A}
P_s(G_i)\right)^\frac{N-t}{N-s}}\\
&>S \frac{\sum_{[i]\in\mathcal A} (sP_s(G_i))^\frac{N-t}{N-s}}{\left(s\sum_{[i]\in\mathcal A}
P_s(G_i)\right)^\frac{N-t}{N-s}}\ge S,
\end{aligned}
 \]
which is impossible. To get the last estimate we used the 
elementary inequality $\left(\sum_i
c_i\right)^\alpha\le\sum_i c_i^\alpha$ 
which holds true for $c_i\ge0$ and $\alpha\in(0,1)$. 

Now, let $j$ be the index
satisfying
\eqref{quasifine}. 
Then, by \eqref{pag3} we get
\begin{equation}\label{8do}
\frac{(1-t)P_t(G_j)}{(sP_s(G_j))^\frac{N-t}{N-s}} \le \liminf_{n\to\infty}
\frac{(1-t)P_t(E_n)}{(sP_s(E_n))^\frac{N-t}{N-s}}\,.
\end{equation}
Then, given any set~$E$, fixed any~$\epsilon>0$, we intersecate~$E$
with a big ball~$B_{R_\epsilon}$ in such a way that
$$ \frac{(1-t)P_t(E\cap B_{R_\epsilon})}{
(sP_s(E\cap B_{R_\epsilon}))^\frac{N-t}{N-s}} \le \frac{(1-t)P_t(E)}{
(sP_s(E))^\frac{N-t}{N-s}}+\epsilon.$$ Then, by the minimality of~$E_n$,
$$ \frac{(1-t)P_t(E\cap B_{R_\epsilon})}{
(sP_s(E\cap B_{R_\epsilon}))^\frac{N-t}{N-s}}\ge
\frac{(1-t)P_t(E_n)}{
(sP_s(E_n))^\frac{N-t}{N-s}}$$
for any~$n\ge n_\epsilon$.
Thus, by~\eqref{8do},
$$ \frac{(1-t)P_t(E)}{
(sP_s(E))^\frac{N-t}{N-s}} +\epsilon\ge
\liminf_{n\to\infty}
\frac{(1-t)P_t(E_n)}{(sP_s(E_n))^\frac{N-t}{N-s}}\ge
\frac{(1-t)P_t(G_j)}{(sP_s(G_j))^\frac{N-t}{N-s}}.$$
By sending~$\epsilon\searrow0$ we see that~$G_j$ is the desired minimizer,
which concludes the proof.
%
%
\end{proof}

\begin{proposition}\label{regfrac}
Let $F$ be a minimizer of \eqref{isoper}. Then $F$ is a 
multiplicative $\omega$-minimizer of the $t$-perimeter, that is, 
for any set $E$ such that $F\Delta E\subset B(x,R)$, there holds 
\[
P_t(F,B(x,R))\leq (1+CR^{t-s})\,P_t(E,B(x,R))\qquad
\text{for any }R<R_0\,,
\]
where $R_0,\,C$ depend only on $N,\,\delta_0$ and $|F|$.
\end{proposition}

\begin{proof} First, if~$\alpha\in(0,1)$,
by graphic the functions, one sees that,
for any~$r\ge0$,
\begin{equation}\label{iN.01}
1-r^\alpha\le |1-r|.
\end{equation}
Also, from~\eqref{P21}, we know that
$$ P_s(F)-P_s(E)\le P_s(F\Delta E),$$
for any sets~$E$ and~$F$, and so, by possibly exchanging the roles
of~$E$ and~$F$ we obtain
\begin{equation}\label{iN.02}
|P_s(E)-P_s(F)|\le P_s(F\Delta E)
\end{equation}
Now, letting $E$ be such that $F\Delta E\subset B(x,R)$, using
the minimality of~$F$, \eqref{iN.01}
and~\eqref{iN.02} we see that
\begin{eqnarray*}
P_t(E)&\ge& P_s(E)^\frac{N-t}{N-s} \frac{P_t(F)}{P_s(F)^\frac{N-t}{N-s}}
\\
&=& P_t(F)+\left(\frac{P_s(E)^\frac{N-t}{N-s}}{P_s(F)^\frac{N-t}{N-s}}-1\right) P_t(F)
\\
&\geq& P_t(F)-\left|\frac{P_s(E)}{P_s(F)}-1\right|P_t(F)
\\
&\geq& P_t(F)-\frac{P_t(F)}{P_s(F)}|P_s(E)-P_s(F)|
\\
&\ge& P_t(F)-\frac{P_t(F)}{P_s(F)} P_s(F\Delta E).
\end{eqnarray*}
Hence,
by applying the fractional isoperimetric inequality \eqref{fracisop} to $P_s(F)$,
we obtain that
$$ P_t(E)\ge P_t(F)-C(N,\delta_0)|F|^{-\frac{N-s}{N}}P_s(F\Delta E).$$
As in \eqref{stimamon}, by means of \eqref{basicestimates} and again the fractional isoperimetric inequality
we then get
\begin{eqnarray*}
P_t(E,B(x,R))&\ge& P_t(F,B(x,R))-C(N,\delta_0) |F|^{-\frac{N-s}{N}} |F\Delta E|^\frac{t-s}{N}
P_t(F\Delta E) 
\\
&=& \left( 1-C(N,\delta_0)|F|^{-\frac{N-s}{N}}R^{t-s}\right)P_t(F,B(x,R))\,,
\end{eqnarray*}
which gives
\[
P_t(F,B(x,R))\leq \frac{|F|^{-\frac{N-s}{N}}}{1-C(N,\delta_0)R^{t-s}}\,P_t(E,B(x,R))\,.
\]
\end{proof}

Reasoning as in Section \ref{section4}, from Proposition \eqref{regfrac}
we obtain the following regularity result.

\begin{corollary}\label{corregtris}
There exists $\beta=\beta(N,\delta_0)<1$ such that any minimizer
$F$ of \eqref{isoper} is bounded and has boundary of class $C^{1,\beta}$,
outside of a closed singular set of Hausdorff dimension at most $N-2$
(respectively $N-8$ if $t=1$).
\end{corollary}

\noindent{\it Proof of Theorem \ref{mainisopintro}.}
The existence claim is a consequence of Theorem~\ref{mainisop}
and the regularity follows from Corollary~\ref{corregtris}.~\qed


\begin{thebibliography}{100}

\bibitem{AFM13} {\sc E. Acerbi, N. Fusco, M. Morini}, Minimality via second variation for a nonlocal isoperimetric
problem. {\it Comm. Math. Phys.}, {\bf 322} (2013), 515--557.

\bibitem{AmbCasMasMor} {\sc L. Ambrosio, V. Caselles, S. Masnou, J.-M. Morel},
Connected components of sets of finite perimeter and applications to image processing. {\it  J. Eur. Math. Soc. (JEMS)},
{\bf 3} (2001), no. 1, 39--92.


\bibitem{ADM11} {\sc L. Ambrosio, G. De Philippis, L. Martinazzi}, $\Gamma$-convergence of nonlocal perimeter
functionals. {\it Manuscripta Math.}, {\bf 134} (2011), 377--403.


\bibitem{bombieri} {\sc E. Bombieri}, Regularity theory for almost minimal currents. {\it Arch. Rational Mech.
Anal.}, {\bf 78} (1982), 99--130.

\bibitem{BC14} {\sc M. Bonacini, R. Cristoferi}, Local and global minimality results for a nonlocal isoperimetric
problem on $\mathbb{R}^N$. To appear in {\it SIAM J. Math. Anal.}.

\bibitem{BBM01} {\sc J. Bourgain, H. Brezis, P. Mironescu}, Another look at Sobolev spaces, {\it  Optimal Control and Partial Differential Equations}, J. L. Menaldi, E. Rofman, A. Sulem (Eds.), a volume in honor of A. 
Bensoussan's 60th birthday, IOS Press, Amsterdam, 2001, 439--455.

\bibitem{Bo} {\sc J. Bourgain, H. Brezis, P. Mironescu},
Limiting embedding theorems for $W^{s,p}$ when $s\nearrow1$ and applications.
{\it J. Anal. Math.}, {\bf 87} (2002), 77--101. 

\bibitem{bralinpar}{\sc L. Brasco, E. Lindgren, E. Parini}, The 
fractional Cheeger problem. {\it Preprint }(2013). Available at 
\url{http://cvgmt.sns.it/paper/2225/}. 

\bibitem{CRS10} {\sc L. Caffarelli, J. M. Roquejoffre, O. Savin}, 
Non-local minimal surfaces. 
{\it Comm. Pure Appl. Math.}, {\bf 63} (2010), 1111--1144.

\bibitem{CV11} {\sc L. Caffarelli, E. Valdinoci},
Uniform estimates and limiting arguments for nonlocal minimal surfaces.
{\it Calc. Var. Partial Differential Equations}, {\bf 41} (2011),
no. 1-2, 203--340. 

\bibitem{CV13} {\sc L. Caffarelli, E. Valdinoci},
Regularity properties of nonlocal minimal surfaces via limiting arguments. 
{\it Adv. Math.}, {\bf 248} (2013), 843--871.

\bibitem{CG11} {\sc M. C. Caputo, N. Guillen},
Regularity for non-local almost minimal boundaries and applications.
{\it Preprint }(2011). Available at \url{http://arxiv.org/pdf/1003.2470}.


\bibitem{CS13} {\sc M. Cicalese, E. Spadaro}, Droplet minimizers of an isoperimetric problem with longe-range
interactions. {\it Comm. Pure Appl. Math.}, {\bf 66} (2013), 1298--1333.

\bibitem{davila} {\sc J. D\'avila},
On an open question about functions of bounded variation.
{\it Calc. Var. Partial Differential Equations}, 
{\bf 15} (2002), no. 4, 519--527.

\bibitem{delpino} {\sc J. D\'avila, M. del Pino, J. Wei},
Nonlocal $s$-minimal surfaces and Lawson cones.
{\it Preprint }(2014). Available at \url{http://arxiv.org/abs/1402.4173}.

\bibitem{DPV12} {\sc E. Di Nezza, G. Palatucci, E. Valdinoci}, Hitchhiker's quide to the frctional Sobolev Spaces.
{\it Bull. Sci. Math.}, {\bf 136} (2012), 521--573.

\bibitem{DFPV} {\sc S. Dipierro, A. Figalli, G. Palatucci, E. Valdinoci},
{Asymptotics of the s-perimeter as~$s\searrow0$},
{\it Discrete Contin. Dyn. Syst.}, {\bf 33} (2013), no. 7, 2777--2790. 


\bibitem{labandadei5} {\sc A. Figalli, N. Fusco, F. Maggi, V. Millot, M. Morini}, Isoperimetry and stability
properties of balls with respect to nonlocal energies. 
{\it Preprint }(2014). Available at \url{http://cvgmt.sns.it/paper/2380/}.

\bibitem{FS}
{\sc R.L. Frank, E.H. Lieb, R. Seiringer}, 
Hardy-Lieb-Thiring inequalities for fractional schr\"odinger operators. 
{\it J. Amer. Math. Soc.}, {\bf 21} (2008), 925--950.

\bibitem{FS08} {\sc R.L. Frank, R. Seiringer}, Non-linear ground state representations and sharp Hardy
inequalities. {\it J. Funct. Anal.}, {\bf 255} (2008), 3407--3430.

\bibitem{F89}{\sc B. Fuglede}, {Stability in the isoperimetric problem for convex or nearly spherical domains in
$\mathbb{R}^n$}. {\it Trans.
Amer. Math. Soc.}, {\bf 314} (1989), 619--638.

\bibitem{FJ} {\sc N. Fusco, V. Julin}, {\em On the regularity of critical and minimal sets of a free
interface problem}, arXiv preprint arXiv:1309.6810 (2013).

\bibitem{FMP08}{\sc N. Fusco, F. Maggi, A. Pratelli}, {The sharp quantitative isoperimetric inequality}. {\it Ann.
of Math.}, {\bf 168} (2008), 941--980.

\bibitem{FMM11}{\sc N. Fusco, V. Millot, M. Morini}, {A quantitative isoperimetric inequality for fractional
perimeters}. {\it J.
Funct. Anal.}, {\bf 26} (2011), 697--715.

\bibitem{GN}
{\sc M. Goldman, M. Novaga},
Volume-constrained minimizers for the prescribed curvature problem in periodic media. {\it Calc. Var. Partial Differential Equations},
{\bf 44} (2012), no. 3, 297--318.

\bibitem{GNR13}{\sc M. Goldman, M. Novaga, B. Ruffini}, {Existence and stability for a non-local isoperimetric
model of charged liquid drops}, 
{\it Preprint} (2013).

\bibitem{J} {\sc V. Julin,} {\em Isoperimetric problem with a Coulombic repulsive term}. To appear in {\it Indiana Univ.
Math. J.}.

\bibitem{KMP13_I}{\sc H. Kn\"upfer, C. B. Muratov}, {On an Isoperimetric Problem with a Competing Nonlocal Term I:
The planar case}. 
{\it Comm. Pure Appl. Math.}, {\bf 66} (2013), 1129--1162.

\bibitem{KMP13_II}{\sc H. Kn\"upfer, C. B. Muratov}, {On an Isoperimetric Problem with a Competing Nonlocal Term
II: The general case}. 
{\it Comm. Pure Appl. Math.}, doi:10.1002/cpa.21479.

\bibitem{M12}{\sc F. Maggi}, {\it Sets of finite perimeter and geometric variational problems. An introduction to
Geometric Measure Theory}. 
Cambridge Studies in Advanced Mathematics {\bf 135}, Cambridge University Press, 2012.

\bibitem{Maz}
{\sc V. Maz'ya, T. Shaposhnikova},
On the Bourgain, Brezis, and Mironescu theorem concerning limiting embeddings of fractional Sobolev spaces.
{\it J. Funct. Anal.}, {\bf 195} (2002), no. 2, 230-238. 

\bibitem{Sam}
{\sc S. G. Samko}, {\it Hypersingular integrals and their applications. Analytical Methods and Special Functions}. 
Taylor \& Francis, Ltd., London, 2002. 

\bibitem{SV12}{\sc O. Savin, E. Valdinoci}, {$\Gamma$-convergence for nonlocal phase transitions}. 
{\it Ann. Inst. H. Poincar\'e Anal. Non Lin\`eaire}, {\bf 29} (2012), 479--500.

\bibitem{SV13}{\sc O. Savin, E. Valdinoci},
Regularity of nonlocal minimal cones in dimension~$2$.
{\it Calc. Var. Partial Differential Equations}, {\bf 48} (2013), no. 1-2, 
33--39. 

\bibitem{tamanini} {\sc I. Tamanini}, 
Boundaries of Caccioppoli sets with H\"older-continuous normal vector. {\it J. Reine Angew. Math.}, {\bf 334}
(1982), 27--39.

\bibitem{Vis} {\sc A. Visintin},
Generalized coarea formula and fractal sets.
{\it Japan J. Indust. Appl. Math.}, {\bf 8} (1991), no. 2, 175-201.




\end{thebibliography}
\end{document}